\documentclass[11pt,oneside,reqno]{amsart}
\usepackage[dvipdfm]{graphicx}
\usepackage{multirow,bigstrut}
\usepackage{amsmath,amsfonts,amssymb,cite}
\usepackage{mathrsfs}
\usepackage{color}
\usepackage{supertabular}
\usepackage{algorithm}
\usepackage{algorithmic}
\usepackage{subfigure}

\usepackage[normalem]{ulem}
\usepackage[top=1.1in, bottom=1.1in, left=1.1in, right=1.1in]{geometry}
\usepackage[pdftex, unicode=true,
            colorlinks,
            linkcolor=blue,
            anchorcolor=blue,
            citecolor=blue
            ]{hyperref}

\usepackage{titlesec}
\titleformat{\section}{\vskip5pt\large\bfseries}{\thesection.}{0.5em}{\centering\vspace{3pt}}
\titleformat{\subsection}{\vskip5pt\normalsize\bfseries}{\thesubsection.}{0.5em}{}

\theoremstyle{definition}

\theoremstyle{plain}
\newtheorem{theorem}{Theorem}
\newtheorem{lemma}{Lemma}
\newtheorem{remark}{Remark}

\numberwithin{theorem}{section}
\numberwithin{lemma}{section}
\numberwithin{remark}{section}
\numberwithin{corollary}{section}
\numberwithin{equation}{section}

\newcommand{\tstd}[1]{e^{#1 \Delta}}

\def\tPN{{P}_{L, N}}
\def\T{\mathbb{T}}
\def\R{\mathbb{R}}
\def\Z{\mathbb{Z}}

\def\NL{NL^{-1}}

\def\M{\mathcal{M}}

\allowdisplaybreaks

\begin{document}

\title[]{Convergence of a moving window method for\\ 
the Schr\"odinger equation with potential on $\mathbb{R}^{\MakeLowercase{d}}$} 

\author[]{Arieh Iserles}
\address{Arieh Iserles: Department of Applied Mathematics and Theoretical Physics, University of Cambridge, E-mail: {\tt A.Iserles@damtp.cam.ac.uk}}

\author[]{Buyang Li}
\address{Buyang Li: Department of Applied Mathematics, The Hong Kong Polytechnic University, Hong Kong. E-mail: {\tt buyang.li@polyu.edu.hk}}

\author[]{Fangyan Yao}
\address{Fangyan Yao: Department of Applied Mathematics, The Hong Kong Polytechnic University, Hong Kong. E-mail: {\tt fangyan.yao@polyu.edu.hk}}

\date{}

\begin{abstract} 
We propose a novel framework, called moving window method, for solving the linear Schr\"odinger equation with an external potential in $\mathbb{R}^d$. This method employs a smooth cut-off function to truncate the equation from Cauchy boundary conditions in the whole space to a bounded window of scaled torus, which is itself moving with the solution. This allows for the application of established schemes on this scaled torus to design algorithms for the whole-space problem. Rigorous analysis of the error in approximating the whole-space solution by numerical solutions on a bounded window is established. Additionally, analytical tools for periodic cases are used to rigorously estimate the error of these whole-space algorithms. 
By integrating the proposed framework with a classical first-order exponential integrator on the scaled torus, we demonstrate that the proposed scheme achieves first-order convergence in time and $\gamma/2$-order convergence in space for initial data in $H^\gamma(\mathbb{R}^d) \cap L^2(\R^d;|x|^{2\gamma} dx)$ with $\gamma \geq 2$. In the case where $\gamma = 1$, the numerical scheme is shown to have half-order convergence under an additional CFL condition. In practice, we can dynamically adjust the window when waves reach its boundary, allowing for continued computation beyond the initial window. Extensive numerical examples are presented to support the theoretical analysis and demonstrate the effectiveness of the proposed method.\\ 


\end{abstract}


\maketitle

\vspace{-15pt}

\section{Introduction}
This paper is concerned with the numerical approximation of solutions to the linear Schr\"odinger equation with an external potential in the whole space of $\mathbb{R}^d$, i.e.,
\begin{align}\label{model-lsp}
\left\{
\begin{aligned}
i \partial_t u(x,t) + \Delta u(x,t) &= V(x)u(x,t) && \quad \mbox{for}\,\,\, (x,t) \in \mathbb{R}^d \times (0,T],
\\
u(x,0) &= u_0(x) && \quad \mbox{for}\,\,\, x \in \mathbb{R}^d,
\end{aligned}
\right. 
\end{align}
where  the solution $u(x,t)$ represents a complex-valued wave function that evolves over time under the influence of the potential $V(x)$, which is assumed to be smooth, compactly supported, and independent of time, and $u_0$ is a given initial wave function. In quantum physics this equation models the time evolution of quantum states, serving as a cornerstone for understanding microscopic particle dynamics. In theoretical chemistry, it is pivotal for simulating molecular interactions and chemical reactions. In optics it describes wave propagation through various media, while in condensed matter physics, it aids in the examination of electronic structures and material behavior. For extensive elaboration of related theoretical developments and applications, see \cite{IKS2018,MS1997PRA,SD1995CMA, Shankar1994}.

One of the main computational challenges in solving this problem numerically involves accurately representing and calculating wave functions over an infinite spatial domain, presenting significant complexities. A related yet more tractable problem is the Schr\"odinger equation on the torus $\mathbb{T}^d$ with $\T=[-\pi,\pi]$, a bounded domain with periodic boundary conditions. This periodic formulation naturally surfaces in scenarios where both the initial condition and the potential are periodic, such as periodic lattice crystal potentials in solid-state physics. Over the last few decades, rigorous  analysis of numerical methods for the periodic Sch\"odinger equation on the torus $\mathbb{T}^d$ has been extensively studied. 

Among the prominent approaches, the time-splitting spectral method \cite{BIK+2014FCM,JL2000BNM,Lubich2008MC} and exponential integrators \cite{Faou2012,HO2010AN} have gained popularity. Both methods leverage the Fourier spectral method for spatial discretization by effectively utilizing the Fourier series which forms a countable orthogonal basis of $L^2(\mathbb{T}^d)$. In terms of accuracy, the Fourier spectral method offers spectral convergence in spatial discretization for smooth solutions 
and has $\gamma/2$-order temporal convergence in the $L^2$ norm for solutions in $C([0,T],H^\gamma)$ with $\gamma \ge 2$. The computation using Fourier spectral method can also be greatly simplified in the frequency domain, thanks to the fact that Fourier bases are eigenfunctions of the Laplacian operator. Additionally, the celebrated Fast Fourier Transform (FFT) accelerates transformations between the physical and frequency domains. 

Recently, a new class of time-integrators, known as low-regularity integrators (LRIs), has been developed in \cite{OS2018FCM,CLL2023IJNA, LW2021NM, MWZ2024, ORS2021FCM, ORS2022JEMS, WY2021MC}. These integrators aim to reduce the regularity requirements of numerical schemes for the nonlinear Schr\"odinger equation while maintaining equivalent convergence rates. The LRIs also utilize the Fourier spectral method for spatial discretization, with a key innovation being their exploitation of resonances among different Fourier modes.

The critical role of the Fourier basis in the above-mentioned methods underscores the importance of having an orthogonal basis in the $L^2(\mathbb{T}^d)$ space for designing numerical schemes for the periodic Schr\"odinger equation, while the primary challenge in numerically solving the Schr\"odinger equation with an external potential in the whole space $\R^d$ lies in accurately representing and computing wave functions over an unbounded spatial domain, which introduces significant computational challenges. 
%

Accordingly, significant efforts have been devoted to research on orthogonal systems on $L^2(\R^d)$, including Hermite functions \cite{BS2005SJSC,Weideman2006JPMG}, Malmquist-Takenaka (M-T) functions \cite{Boyd1987JCP,Christov1982SJAM,Weideman1994SJNA}, and stretched Fourier functions \cite{DITRN}. 
For interested readers, we recommend \cite{ILW2022,IW2019FCM} for more details. 
The stretched Fourier functions have the simplest construction, with basis functions being zero extensions of the Fourier basis on a scaled torus. 
In contrast, the basis functions for Hermite and M-T functions are derived from applying the scaled inverse Fourier transform to orthogonal polynomials in frequency space \cite{IW2020JFAA}. 
Recently, the M-T functions have been applied to solve the semiclassical linear Schr\"odinger equation with initial functions in the form of Gaussian packets \cite{IKS+2022}.
As pointed out in \cite{ILW2022}, the M-T functions outperform other methods in terms of accuracy and efficiency for approximating Gaussian wave packets. 
It is proved that the M-T expansion coefficient of the Gaussian wave packets decays exponentially, and can be approximated by utilizing a relation between the M-T system and the Fourier series on $\T$, i.e., 
\begin{align}\label{MTfft}
(u, \phi_n)=\frac{(-i)^n}{2 \sqrt{2 \pi}} \int_{-\pi}^\pi U( \theta) \mathrm{e}^{-i n \theta} \mathrm{d} \theta,
\quad 
U(\theta)=\Big(1-i \tan \frac{\theta}{2}\Big) u \Big(\frac{1}{2} \tan \frac{\theta}{2}\Big),
\end{align} 
where $ \phi_n$ is the $n$th normalized orthogonal basis function in the M-T system. 

However, applying the M-T system to the linear Schr\"odinger equation with an external potential, particularly when the exact solution exhibits limited regularity, still presents challenges in achieving optimal-order convergence in space. Let us explain this point by examining the error of approximating a function by its low-frequency projection under different systems, such as the M-T system and the stretched Fourier system.

For the M-T system, according to \eqref{MTfft}, the associated low-frequency projection ${P^1_N} u$ satisfies
\begin{align*}
{P^1_N} u = \sum_{|n|\le N} a_n\phi_n 
\quad\mbox{with}\,\,\, a_n = \frac{(-i)^n}{2 \sqrt{2 \pi}} \int_{-\pi}^\pi U( \theta) \mathrm{e}^{-i n \theta} \mathrm{d} \theta . 
\end{align*}
Thus the error $u - {P^1_N} u \in L^2(\R)$ is related to $U-P_NU \in L^2(\T)$ through coordinate transformation. This leads to the following error estimate: 
\begin{align}\label{interp-error}
\| u - {P^1_N} u\|_{L^2(\R)}
\lesssim \|U - P_N U\|_{L^2(\T)}\lesssim N^{-1} \| U \|_{H^1(\T)}.
\end{align}
By calculation it is easy to find that 
\begin{align*}
\partial_\theta U( \theta)
= -\frac{i}{2}(1+4x^2) u(x) + 
\frac14(1-2ix)\partial_xu(x)(1+4x^2)
\quad\mbox{with}\,\,\, x = \frac12 \tan( \theta/2)
\end{align*}
and therefore 
\begin{align}\label{U-H1}
\| U \|_{H^1(\T)}\lesssim
\|u\|_{H^1(\R)}+ \| x u\|_{L^2(\R)}+ \| x^2 \partial_xu\|_{L^2(\R)},
\end{align}
which depends on the rates of decay of $u$ and $\partial_xu$ as $x\rightarrow\infty$. As will be shown in this paper, the problem \eqref{model-lsp} is well-posed in \( H^\gamma(\R^d) \cap L^2(\R^d;|x|^{2\gamma} \, dx) \). In view of \eqref{interp-error} and \eqref{U-H1}, first-order convergence in approximating the solution $u$ generally requires $\partial_xu\in L^2(\R^d;|x|^{4} \, dx) $ and therefore requires $u \in H^\gamma(\R^d) \cap L^2(\R^d;|x|^{2\gamma} \, dx) \hookrightarrow H^1(\R^d;|x|^{2\gamma-2} \, dx)$ with $ \gamma \geq 3 $. 
More generally, $\gamma/2$-order convergence in $L^2(\R^d)$ cannot be achieved for $u \in H^\gamma(\R^d) \cap L^2(\R^d;|x|^{2\gamma} \, dx) $.

For the stretched Fourier system which utilizes the zero extensions of the Fourier series on the scaled torus $[-L,L]$ by approximating a function $u$ with $P_N^2 u$, defined as the zero extension to $\R$ of the $N$-term Fourier series of $u_L = u|_{(-L,L)}$ on $[-L,L]$, the following error estimate holds (according to the Bernstein inequality in Lemma~\ref{lem:Bernstein}): 
\begin{align}\label{u-P2Nu}
\| P_N^2 u - u\|_{L^2(\R)}
&\lesssim 
\| u\|_{L^2((-L,L)^c)} + 
\| P_N^2 u_L - u_L\|_{L^2(-L,L)} \notag\\ 
&\lesssim L^{-1} \| xu\|_{L^2(\R)} + LN^{-1} \| u_L\|_{H^1_{\textrm{per}}([-L,L])} ,
\end{align}
where $H^1_{\textrm{per}}([-L,L])$ denotes the $H^1$ norm for periodic functions on $[-L,L]$. 
Though the first term on the right-hand side of \eqref{u-P2Nu} can now be well controlled by the decaying property of the solution $u$, the second term may not be bound as $u_L$ may fail to be periodic. Thus, the challenge of designing a numerical scheme for \eqref{model-lsp} in the whole space with provable optimal convergence order remains unresolved.

The objective of this paper is to develop a general framework for solving the linear Schr\"odinger equation with an external potential in the unbounded domain $\mathbb{R}^d$, with optimal-order convergence for solutions in $H^\gamma(\R^d) \cap L^2(\R^d;|x|^{2\gamma} dx)$ with $\gamma \in\Z^+$. Our approach is to incorporate a smooth domain-truncation function \(\chi_L\) into the framework of stretched Fourier system, to derive a truncated partial differential equation (PDE) on a bounded window of scaled torus $\Omega_L=[-L,L]^d$ and then solve the PDE by using standard approaches such as exponential integrators for PDEs on a torus. 
The smooth domain-truncation function \(\chi_L\) is constructed by dilating a cut-off function defined on \((-1,1)^d\), such that \(\chi_L\) is supported on \([-L, L]^d\) and equals 1 on \((-aL, aL)^d\) for some $a\in(0,1)$. 
Then we decompose the solution of problem \eqref{model-lsp} into two parts: \(u_L = \chi_L  u |_{\Omega_L}\) and \(u - u_L\), where the latter represents the domain-truncation error that is easily shown to be of an optimal order.
The first part is periodic on $\Omega_L=[-L,L]^d$ and satisfies a modified equation 
\begin{align}\label{eq:truncate}
i \partial_t u_L + \Delta u_L
&= \chi_L V u_L + R , 
\end{align}
with a remainder $R$ that can be shown to satisfy the following estimate (thanks to the scaling property involved in the construction of $\chi_L$):
\begin{align*}
\|R\|_{L^\infty(0,T;L^2(\Omega_L))} \le C L^{-\gamma}.
\end{align*}
Thus, the proposed numerical scheme consists of truncating the remainder $R$ and applying an exponential integrator to \eqref{eq:truncate} by utilizing Fourier series on the scaled torus $\Omega_L=[-L,L]^d$. 
For illustration, we present rigorous analysis for the convergence of an exponential Euler method in this framework. By choosing an appropriate frequency \(N = L^2\), we prove that the method achieves first-order convergence in time and \(\gamma/2\)-order convergence in space for \(u_0\in H^\gamma(\R^d) \cap L^2(\R^d;|x|^{2\gamma} dx)\) with \(\gamma \geq 2\). In the case \(\gamma = 1\), an additional CFL condition \(\tau=O(N^{-1})\) is required for the algorithm to have half-order convergence. 


From a practical perspective, it is crucial to address the spurious reflections that occur when the waves approach the boundary of the bounded window of computational domain. To mitigate this issue, we implement a dynamic adjustment of the truncation parameter \(L\) by detecting when the solution nears the boundary. When this occurs, we implement a reinitialization procedure that involves extending the bounded window of computational domain and applying zero-extension to the numerical solutions in physical space.

The rest of this paper is organized as follows: In Section \ref{section:notation}, we introduce the notation and present the main results of this paper, including the numerical scheme and the main theoretical results on the convergence of numerical solutions. In Section \ref{section:preliminary}, we present some basic theoretical results and tools for the error analysis in this paper, including well-posedness of the Schr\"odinger equation in a weighted Sobolev space, Bernstein's inequalities, and the Kato--Ponce inequalities. In Section \ref{section:truncated}, we present the proof of the main theorem, including error estimates for approximating functions on the whole space in a weight Sobolev space, and error estimates for the truncated problem on the scaled torus. Numerical examples are presented in Section \ref{section:numerical}, and conclusions are presented in Section \ref{section:conclusion}. 

The eventual evolution of our approach, underscored by its name, `a moving window method', should allow for the parameter $L$ to vary with time: grow once the solution nears the boundary, shrink once it recedes away from it. Moreover, in principle, the torus need not be of the same size in each dimension, i.e.\ we may consider a periodization of the parallelepiped 
\begin{displaymath}
  [-L_1,L_1]\times[-L_2,L_2]\times\cdots\times[-L_d,L_d].
\end{displaymath}
This will feature in future publication, while in this paper we focus on the mathematics underlying our construction.

\section{Notation and main results}\label{section:notation}
In this section, we introduce the main concepts and the notation that will be used throughout this paper, and then present our numerical scheme and the convergence results for the linear Schr\"odinger equation with a potential on $\R^d$.

\subsection{Notation}

The Fourier transform of a function $f\in L^2(\R^d)$ is denoted by ${\mathcal{F}} f(\xi) = \frac{1}{(2\pi)^d}\int_{\R^d} f(x) e^{-ix \cdot \xi} dx$. 
The Bessel potential operator $J^s$ is defined as 
\begin{align*}
J^s f = \mathcal{F}_{\xi}^{-1} [(1+ |\xi|^2)^{s/2} \mathcal{F}f(\xi)] , 
\end{align*}
where $\mathcal{F}_{\xi}^{-1}$ denotes the inverse Fourier transform with respect to $\xi$.  
Similarly, we define 
\begin{align*}
|\nabla |^s f = \mathcal{F}_{\xi}^{-1} [|\xi|^s \mathcal{F}f(\xi)] , 
\end{align*} 
For a function $ f $ defined on the torus $\Omega_L=[-L,L]^d $, we define its Fourier transform on $\Omega_L$ as
\begin{equation}\label{defPeriodFourier}
\mathcal{F}_L f(k) = \frac{1}{(2L)^d} \int_{\Omega_L} f(x) e^{-i \pi k \cdot x/L} \, dx.
\end{equation}
Then $f(x) =  \sum_{k\in\Z^d} {\mathcal{F}}_L f(k) e^{ i \pi k\cdot x/L}$ and the following Parseval's identity holds:
\begin{align}\label{Parseval}
&\| f\|_{L^2( \Omega_L)}^2 = (2L)^d \| {\mathcal{F}} _L f(k)\|_{\ell^2_k}^2
 = (2L)^d \sum_{k\in \Z^d} | {\mathcal{F}} _L f(k)|^2 . 
\end{align} 
For $k\in \Z^d$, we denote by $|k|_\infty$ the maximal absolute value of its components. We define $J_L^s$ by its Fourier coefficients in a manner similar to that of its continuous counterpart, 
\begin{align*}
J^s_L f = {\mathcal{F}} _{L,k}^{-1} [(1+ |k|^2\pi^2/L^2)^{s/2} {\mathcal{F}}_L f(k)] ,
\end{align*}
where $\mathcal{F}_{L,k}^{-1}$ denotes the inverse Fourier transform on $\Omega_L$ with respect to $k$. Therefore, using relation \eqref{defPeriodFourier}, we have 
\begin{align}\label{FT-Jsf}
\| J^s_L f \|_{L^2(\Omega_L)} = (2L)^{d/2} \| (1+ |k|^2\pi^2/L^2)^{s/2} {\mathcal{F}}_L f(k) \|_{l^2_k} .
\end{align}
where $\|g(k)\|_{l^2_k}: = \big(\sum_{k\in\Z^d} |g(k)|^2\big)^{\frac12}$ for any square summable function $g$ defined on $\Z^d$. 
Similarly, we define
\begin{align}\label{def:nab}
\| |\nabla|^s_L f \|_{L^2(\Omega_L)} = (2L)^{d/2} \| (|k|\pi/L)^{s} {\mathcal{F}}_L f(k) \|_{l^2_k} .
\end{align} 
Then we introduce  periodic Sobolev spaces 
\begin{align*}
H^s_{\mathrm{\mathrm{per}}}( \Omega_L):= 
\{ 
f\in L^2( \Omega_L): J^s_L f\in L^2( \Omega_L) 
\} . 
\end{align*}
For simplicity of notation, we shall omit the subscript in $H^s_{\mathrm{per}}( \Omega_L)$, writing $H^{s}( \Omega_L)$ instead.

The low-frequency projection $\tPN: L^2(\Omega_L)\rightarrow S_{L,N}=\big\{\sum_{|k|_\infty\le  N} a_k  e^{ i \pi k\cdot x/L}: a_k\in\mathbb{C}\,\,\mbox{for}\,\,k\in\Z^d\big\}$ is defined as 
\begin{align*} 
\tPN f =  \sum_{|k|_\infty\le  N} {\mathcal{F}}_L f(k) e^{ i \pi k\cdot x/L}. 
\end{align*} 
The trigonometric interpolation of a continuous function $f$ on $\Omega_L=[-L,L]^d$ is defined as a trigonometric polynomial 
$$
\displaystyle I_{L,N} f=\sum\limits_{|k|_\infty\le  N} a_k  e^{ i \pi k\cdot x/L}
$$ 
satisfying relation 
$I_{L,N} f(x_n) = f(x_n)$ at $x_n=\frac{2nL}{2N+1}$ for $|n|_\infty\le N$. 
In particular, the coefficients $a_k$ in this expression are given by
\begin{align*}
 a_k=\frac{1}{(2N+1)^d}\sum_{|n|_\infty\le  N} f(x_n) e^{ -i \pi k\cdot x/L},
\quad 
 x_n = \left(\frac{2n_1L}{2N+1}, \frac{2n_2L}{2N+1},\cdots \frac{2n_dL}{2N+1}\right).
\end{align*}
The pointwise evaluation of $f(x_n)$ requires $f$ to be continuous on $\Omega_L$. According to the Sobolev embedding theorem, the continuity of $f$ is guaranteed if $f\in H^s(\Omega_L)$ for some $s > d/2$.


For the simplicity of notation, we denote by $ A \lesssim B $ or $ B \gtrsim A $  the statement $ A \leq CB $ for some positive constant $ C $.  The value of $ C $ may vary in different instances, but it is always independent of the stepsize $ \tau $, the degrees of freedom $ N $, diameter of the computational region $L$, and the time level $n$. Additionally, if a statement includes $s+$ or $s-$ for some number $s$, it means the statement holds for $s + \epsilon$ or $s - \epsilon $ for any small $ \epsilon > 0 $.

\subsection{Main results}
Given $ a \in (0,1) $, we define a smooth cut-off function $ \chi_{a,1} \in C_c^\infty(\mathbb{R}^d) $ by setting $ \chi_{a,1}(x) = 1 $ for $ x \in (-a,a)^d $ and ensuring $ \textrm{supp}( \chi_{a,1}) \subset (-1,1)^d $. By scaling, we construct the cut-off function $ \chi_{a,L}(x) = \chi_{a,1}\left(\frac{x}{L}\right) $, which is supported in $ \Omega_L = (-L,L)^d $.

Let $t_n=n\tau$, $n=0,1,\dots,M$, be a partition of the time internal $[0,T]$ with stepsize $\tau=T/M$, and let $u_L^n$ represent the numerical solution at $t=t_n$. We consider the following numerical scheme:
\begin{align}\label{scheme-p}
u_L^{n+1} = \tstd{i \tau} u_L^n -i \tau \varphi_1(i \tau \Delta) \tPN ( I_{L,N}(\chi_{1/2,L} V)\cdot  u_L^n),\quad 
\varphi_1(z) = \frac{e^z -1}{z},
\end{align}
with the initial function $u^0_L = \tPN (\chi_{1/2,L} u_0)$. 

The main theoretical result of this paper is the following theorem. 

\begin{theorem}\label{theorem-lsp}
For given $T>0$ and initial value $u_0\in H^\gamma(\R^d)\cap L^2(\R^d;|x|^{2\gamma}dx)$ with 
$ \gamma\in \Z^+$,
we denote by $u$ the solution of \eqref{model-lsp} and $u_L^n$ the numerical solution given in \eqref{scheme-p}.
Then there exist constants $\tau_0\in (0,1)$ and $C_T>0$ such that the following error estimates hold for $\tau\leq\tau_0$: 
\begin{align*}
&\|u(t_n,\cdot)- Eu_L^n\|_{L^2(\R^d)}
  \leq C_T\big( L^{-1} + (NL^{-1})^{- 1} +\tau NL^{-1} +\tau\big) && \!\!{\rm if}\; \gamma=1\;\mbox{\rm and}\; \tau N^2L^{-2}\le 1,\\
 &\|u(t_n,\cdot)- Eu_L^n\|_{L^2(\R^d)}
  \leq C_T\big( L^{-\gamma}  + (NL^{-1})^{- \gamma} +\tau \big) && \!\!{\rm if} \ \gamma\ge2,
\end{align*}
where $E$ denotes the zero extension from $\Omega_L$ to $\R^d$, and the constants $\tau_0$ and $C_T$ depend only on $\|u_0\|_{H^\gamma(\R^d)\cap L^2(\R^d;|x|^{2\gamma}dx)}$ and $T$.
\end{theorem}

\begin{remark}\upshape 
By choosing $N = L^2$ in the numerical scheme in \eqref{scheme-p}, the error estimates in Theorem \ref{theorem-lsp} reduce to the following results: 
\begin{align*}
\begin{aligned}
&\|u(t_n,\cdot)- Eu_L^n\|_{L^2(\R^d)}
  \leq C_T N^{-\frac12} && {\rm if} \ \gamma=1\;\mbox{\rm and}\; \tau = O(N^{-1}),\\
&\|u(t_n,\cdot)- Eu_L^n\|_{L^2(\R^d)}
  \leq C_T\big( N^{-\frac{\gamma}2}   +\tau \big) && {\rm if} \ \gamma\ge2.
  \end{aligned}
\end{align*}
\end{remark}

\section{Preliminary results}\label{section:preliminary}
In this section we present some preliminary results to be used in the proof of Theorem~\ref{theorem-lsp}. These include the well-posedness of equation \eqref{model-lsp}, Bernstein's inequalities on the scaled torus $ \Omega_L$, and approximation theory (i.e., approximating functions on $\R^d$ by zero extensions of finite Fourier series on the torus $\Omega_L$).

\subsection{Well-posedness of the PDE}
The well-posedness of equation \eqref{model-lsp} in a weighted Sobolev space, which guarantees both regularity and spatially decaying property of the solution, is presented in the following lemma. The proof of the lemma is presented in \ref{appendix}. 

\begin{lemma}[Global well-posedness] \label{lem:lwp}
Let $u_0\in H^\gamma(\R^d)\cap L^2(\R^d;|x|^{2\gamma}dx)$ for some $\gamma \in \Z^+$. Then equation \eqref{model-lsp} has a unique solution $u\in C(\R; H^\gamma(\R^d)) \cap C (\R; L^2(\R^d;|x|^{2 \gamma} d x))$. 
\end{lemma}

\subsection{Bernstein's inequality}

It is known that Bernstein's inequality holds on $\mathbb{R}^d$ and the canonical torus, as shown {\em inter alia\/} in \cite[(A.2)-(A.6)]{Tao2006} and \cite[Theorem 2.2 and pp. 22]{Guo1998}. The results are summarized in the following lemma and will be used in the proof of Theorem \ref{theorem-lsp}. 

 \begin{lemma}[Bernstein's inequalities] \label{bound}
Let $s\ge0$.
For $f\in H^{ s}(\R^d)$, it holds that
\begin{align*}
& \|P_{\le N}J^s f\|_{L^2(\R^d)}  \lesssim N^{s}  \| f\|_{L^2(\R^d)},\qquad
 \|(1-P_{\le N}) f\|_{L^2(\R^d)} \lesssim N^{-s} \|J^{s}  f\|_{L^2(\R^d)}.
\end{align*}
 Moreover, for $f\in H^{s}( \Omega_1)$ we have
 \begin{align*}
 & \|P_{1,N} J^s_1 f\|_{L^2(\Omega_1)}  \lesssim {N}^{s}  \| f\|_{L^2(\Omega_1)},\qquad
 \|(1-P_{1,N}) f\|_{L^2(\Omega_1)} \lesssim {N}^{-s} \|J_1^{s}  f\|_{L^2(\Omega_1)}.
\end{align*}
 \end{lemma}

To prove the Bernstein's inequality on the scaled torus, we first introduce a change of variable in space. 
We define $ m_L: \Omega_1 \to \Omega_L$ by
\begin{align*}
m_L(y) = Ly.
\end{align*}
The following lemma can be proved by direct computation and therefore its proof is omitted.
It reveals that the dilation only changes the Fourier bases, while keeping the generalized Fourier coefficients invariant. 
\begin{lemma}\label{lem:Fourier}
For $k\in \mathbb{Z}^d$ and $f\in L^2( \Omega_L)$, we have
$[{\mathcal{F}} _1 (f\circ m_L)](k) = [{\mathcal{F}} _L f](k)$ .
\end{lemma}

\begin{lemma}\label{lem:sp-scaling}
For $L\ge 1$, we have $\| J_1^s (f\circ m_L)\|_{L^2( \Omega_1)} \lesssim  L^{s-d/2} \| J_L ^s f\|_{L^2( \Omega_L)}$.
\end{lemma}
\begin{proof}
By Parseval's equality \eqref{Parseval} and Lemma~\ref{lem:Fourier}, we obtain
\begin{align*}
\| J_1^s (f\circ m_L)\|_{L^2( \Omega_1)} 
&= 2^{\frac d2}\| { (1+ |k|^2\pi^2)^{\frac s2}\mathcal{F}} _1(f\circ m_L)(k)\|_{l^2}\\
&\le 2^{\frac d2}\| { L^s(1+ (|k|\pi^2/L)^2)^{\frac s2}\mathcal{F}} _L(f)(k)\|_{l^2}
= L^{s-\frac d2} \| J_L^s f\|_{L^2( \Omega_L)} , 
\end{align*}
where the last equality follows from \eqref{FT-Jsf}. 
\end{proof}

%

\begin{lemma}[Approximation result and Bernstein's inequality on $ \Omega_L$]\label{lem:Bernstein}
For $L\ge 1$ and $f\in H^s ( \Omega_L)$ with $s\ge m$, we have
\begin{align*}
 \| f - \tPN f\|_{L^2( \Omega_L)} &\lesssim ({N}L^{-1})^{-s} \| f\|_{H^s( \Omega_L)},\\
 \| \tPN f\|_{H^s( \Omega_L)} &\lesssim \max \{1, ({ N}L^{-1})^{s-m}\}  \| \tPN f\|_{H^m( \Omega_L)}.
\end{align*}
\end{lemma}

\begin{proof}
The first result can be obtained by applying Lemmas~\ref{bound}--\ref{lem:sp-scaling}, which imply that 
\begin{align*}
\| f - \tPN f\|_{L^2( \Omega_L)}  
&= L^{\frac d2}\| f\circ m_L - P_{1,N} (f\circ m_L)\|_{L^2( \Omega_1)} 
	\lesssim L^{\frac d2} {N}^{-s} \| f\circ m_L\|_{H^s( \Omega_1)}\\
&\le (L^{-1} {N})^{-s} \| f\|_{H^s( \Omega_L)}.
\end{align*}
The second result follows from Parseval's equality, i.e., 
\begin{align*}
\| \tPN f\|_{H^s( \Omega_L)} 
&= (2L)^{d/2} \left\|\sum_{|k|_{\infty}\le N}
	\left(1+ \Big(\frac{\pi}{ L}\Big)^2 |k|^2 \right)^{s/2}  {\mathcal{F}} _L f(k) \right\|_{l^2_k} \\
&\lesssim (2L)^{d/2} (1+ N/ L)^{s-m}  \left\|\sum_{|k|_{\infty}\le N}
	\left(1+ \Big(\frac{\pi}{ L}\Big)^2 |k|^2 \right)^{m/2}  {\mathcal{F}} _L f(k) \right\|_{l^2_k} \\
&\lesssim 
(1+ N/ L)^{s-m} \| \tPN f\|_{H^m( \Omega_L)} . 
\end{align*}
This completes the proof of Lemma \ref{lem:Bernstein}. 
\end{proof}
\begin{remark}\label{rem:bern}
\upshape 
The following approximation results can be obtained in the same way: 
\begin{align*}
 \| f - \tPN f\|_{L^2( \Omega_L)} &\lesssim ( NL^{-1})^{-s} \| |\nabla|_L^s f\|_{L^2( \Omega_L)},\\
  \| f - I_{L,N} f\|_{L^2( \Omega_L)} &\lesssim ( NL^{-1})^{-s} \|  f\|_{H^s( \Omega_L)}, \quad s>d/2.
 \end{align*}
\end{remark}

\subsection{Kato--Ponce inequalities}
The following version of the Kato--Ponce inequalities will be used in the proof of Theorem \ref{theorem-lsp}. 
\begin{lemma} [The Kato--Ponce inequality \cite{kato_commutator_1988, grafakos_kato-ponce_2014}]\label{lem:kato-Ponce} 
For any given $\gamma>0$, the following inequalities hold for $ \Omega =  \Omega_L$ and $\R^d${\rm:}
\begin{align*}
\|J^\gamma (fg)\|_{L^2( \Omega)}&\lesssim 
\| J^ \gamma f\|_{L^\infty ( \Omega)} \| g\|_{L^2( \Omega)} + 
\| f\|_{L^\infty( \Omega)} \| J^ \gamma g\|_{L^2( \Omega)},\\
\||\nabla|^\gamma (fg)\|_{L^2( \Omega)}&\lesssim 
\| |\nabla|^\gamma f\|_{L^\infty ( \Omega)} \| g\|_{L^2( \Omega)} + 
\| f\|_{L^\infty( \Omega)} \| |\nabla|^\gamma g\|_{L^2( \Omega)}.
\end{align*}
\end{lemma}

\medskip

\section{Proof of Theorem \ref{theorem-lsp}}\label{section:truncated}

The proof of Theorem \ref{theorem-lsp} is divided into four subsections, which reflect the fundamental framework for the analyzing such problems.
 
\subsection{Approximation result}
The following theorem demonstrates the accuracy of approximating a function in $\R^d$ by finite Fourier series on the scaled torus $ \Omega_L$.

\begin{theorem}\label{thm:approx}
There exists a constant $C$ independent of $L$ and $a$ such that the following estimate holds for any $f\in H^ \gamma(\R^d) \cap L^2(\R^d;|x|^{2 \gamma} d x)${\rm:}
\begin{align*}
\| f - E\tPN \chi_{a,L} f\|_{L^2(\R^d)} 
&\le 
C(aL)^{- \gamma} \| |x|^ \gamma f\|_{L^2( \R^d)} 
	+ C (1+ ((1-a)L)^{- \gamma})(\NL)^{- \gamma}  \| f\|_{H^{ \gamma}(\R^d)},
\end{align*}
where $E$ denotes the zero-extension from $L^2( \Omega_L)$ to $L^2(\R^d)$.
\end{theorem}
\begin{proof}
By the triangle inequality and Bernstein's inequality in Remark \ref{rem:bern} on the scaled torus, we obtain
\begin{align*}
\|f- E\tPN \chi_{a,L} f\|_{L^2(\R^d)} 
&\le 
\|(\chi_{a,L} - 1) f\|_{L^2(\R^d)} 
+ \| (1-\tPN)\chi_{a,L} f\|_{L^2( \Omega_L)} \\
&\le
\|(\chi_{a,L} - 1) f\|_{L^2( \Omega_{aL}^ c)}
+ (\NL)^{- \gamma} \| |\nabla|_L^ \gamma (\chi_{a,L} f)\|_{L^2( \Omega_L)}.
\end{align*}
Since $1\le (aL)^{-1} |x|$ for $x\in \Omega_{aL}^ c$, it follows that 
\begin{align*}
\|(\chi_{a,L} - 1) f\|_{L^2( \Omega_{aL}^ c)}\le C(aL)^{- \gamma} \| |x|^ \gamma f\|_{L^2( \R^d)}.
\end{align*}
By using \eqref{def:nab}, we have
\begin{align*}
\| |\nabla|_L^ \gamma  (\chi_{a,L} f) \|_{L^2( \Omega_L)}^2
&= (2L)^d(\pi/L)^{2\gamma}\sum_k \Big(\sum_{i = 1}^d{k_i^2}\Big)^ \gamma |{\mathcal{F}}_L f(k)|^2\\
&\lesssim 
(2L)^d(\pi/L)^{2 \gamma}\sum_k 
\sum_{i = 1}^d{k_i^{2 \gamma}} |{\mathcal{F}}_L  (\chi_{a,L} f)(k)|^2\\
&= 
\sum_{i = 1}^d \| \partial_i^ \gamma  (\chi_{a,L} f)\|_{L^2( \Omega_L)}^2
\lesssim 
\sum_{i = 1}^d \| \partial_i^ \gamma  (\chi_{a,L} f)\|_{L^2( \mathbb{R}^d)}^2
\lesssim \| |\nabla|^ \gamma  (\chi_{a,L} f)\|_{L^2( \mathbb{R}^d)}^2 , 
\end{align*}
where the second to last inequality holds because $\gamma$ is a positive integer. 
By further employing the Kato--Ponce inequality in Lemma \ref{lem:kato-Ponce}, we obtain
\begin{align*}
&(\NL)^{- \gamma} \| |\nabla|_L^ \gamma (\chi_{a,L} f)\|_{L^2( \Omega_L)} 
\le (\NL)^{- \gamma} \| |\nabla|^ \gamma (\chi_{a,L} f)\|_{L^2( \R^d)} \\
&\le C
(\NL)^{- \gamma} \Big(\| |\nabla|^ \gamma \chi_{a,L}\|_{L^{\infty}(\R^d)} 
 \| f\|_{L^{2}(\R^d)}
 +\|\chi_{a,L}\|_{L^{\infty}(\R^d)} \| |\nabla|^ \gamma f\|_{L^{2}(\R^d)}
\Big).
\end{align*}
By the scaling property, we have 
\begin{align*}
\| |\nabla|^\gamma \chi_{a,L}\|_{L^{\infty}(\R^d)}\lesssim ((1-a)L)^{- \gamma}.
\end{align*} 
This, together with the Sobolev embedding result $H^{ 0+}\hookrightarrow L^{2+}$, implies that 
\begin{align*}
(\NL)^{- \gamma} \| |\nabla|_L^ \gamma (\chi_{a,L} f)\|_{L^2( \Omega_L)}
&\le C(1+ ((1-a)L)^{- \gamma}) (\NL)^{- \gamma}  \| J^ \gamma f\|_{L^{2}(\R^d)},
\end{align*}
This completes the proof of Theorem \ref{thm:approx}. 
\end{proof}

\begin{remark}
\upshape 
In particular, by choosing $L = N^{1/2}$, we obtain
\begin{equation*}
\| f - E\tPN \chi_{a,L} f\|_{L^2(\R^d)} 
\le 
CN^{-\frac \gamma 2} 
\Big(a^{- \gamma} \| |x|^ \gamma f\|_{L^2( \R^d)}
	+ \big(1+ ((1-a)N^{\frac12})^{- \gamma}\big)\| f\|_{H^{ \gamma}(\R^d)}\Big).
\end{equation*}
\end{remark}

\subsection{Truncated problem and remainders}\label{section:remainders}
In this section, we first derive the equation satisfied by \( u_L = \chi_L u|_{\Omega_L} \),
where $\chi_L = \chi_{1/2,L}$ for simplicity of  notation. 
This equation takes the form of a periodic Schrödinger-type equation with a perturbation term. For simplicity, we will omit the notation for domain confinement. For example, \( \chi_L u \) will refer either to the function in \( L^2(\mathbb{R}^d) \) or its restriction to \( L^2(\Omega_L) \), depending on the context. By direct calculation, we derive
\begin{align*}
i \partial_t (\chi_L u) + \Delta (\chi_L  u)
&= \chi_L (i \partial_t  u + \Delta u) + 2  \nabla \chi_L  \nabla u
+  \Delta \chi_L u \\
&= \chi_L Vu + 2 \nabla \chi_L \nabla u + \Delta\chi_L u,
\end{align*}
where the last line is obtained by utilizing equaiton \eqref{model-lsp}.
Thus, $u_L=\chi_L u|_{\Omega_L}$ satisfies the following Schr\"odinger-type equation on $ \Omega_L$: 
\begin{align}\label{PDE:boundeddomain}
i \partial_t u_L+ \Delta u_L =\chi_L V u_L+ R,\quad 
u_L(0) = \chi_L u_0,
\end{align}
with periodic boundary condition, where 
\begin{align}\label{def:r}
R = \chi_L Vu- \chi_L^2 V u+ 2 \nabla \chi_L  \nabla u + \Delta\chi_L  u.
\end{align}
By Duhamel's formula, the solution $u_L$ of equation \eqref{PDE:boundeddomain}  satisfies 
\begin{eqnarray}\label{equ:solu}
u_L(t_{n+1}) &=& \tstd{ i\tau} u_L(t_n)
- i \int_{t_n}^{t_{n+1}} \tstd{i(t_{n+1}-s)} \big[\chi_L V u_L(s) + R(s) \big] ds\notag\\
&=& \tstd{ i\tau} u_L(t_n)
- i \int_0^\tau \tstd{i(\tau -s)} \big[\chi_L V u_L(t_n+s)\big] ds\nonumber\\
&&\mbox{}- i \int_{t_n}^{t_{n+1}} \tstd{i(t_{n+1}-s)}  R(s) ds.
\end{eqnarray}
Applying the high- and low-frequency decomposition results in
\begin{align*}
u_L(t_{n+1}) &= \tstd{ i\tau} u_L(t_n)
- i \int_0^\tau \tstd{i(\tau -s)} \tPN \big[I_{L,N}(\chi_L V) \tPN u_L(t_n+s)\big] ds- R_1^n-R_2^n,
\end{align*}
where the remainders $R_i^n$, $i=1,2$, are defined as
\begin{eqnarray*}
R_1^n&=&i \int_{t_n}^{t_{n+1}} \tstd{i(t_{n+1}-s)}  R(s) ds,\\
R_2^n&=&i \int_0^\tau \tstd{i(\tau -s)} \big[\chi_L V u_L(t_n+s)\big] ds\\
   &&\mbox{}- i \int_0^\tau \tstd{i(\tau -s)} \tPN \big[I_{L,N}(\chi_L V) \tPN u_L(t_n+s)\big] ds.
\end{eqnarray*}
 Applying the approximation $u_L(t_n + s) \approx u_L(t_n)$ of the integrand and integrating out
$\tstd{i(\tau -s)}$ with respect to $s$, we obtain
\begin{align*}
u_L(t_{n+1}) &= \tstd{ i\tau} u_L(t_n) 
   -i \tau \varphi_1(i \tau \Delta) \tPN ( I_{L,N}(\chi_L V)\cdot  u_L(t_n))-\sum_{i=1}^3R_i^n, 
\end{align*}
where $\varphi_1(z) = \frac{e^z -1}{z}$ and 
\begin{align*}
 R_3^n&= i \int_0^\tau \tstd{i(\tau -s)} \tPN \Big[I_{L,N}(\chi_L V)\cdot \tPN \big(u_L(t_n+s)-u_L(t_n)\big)\Big] ds.
\end{align*}

\subsection{Consistency errors}\qquad\\
\begin{lemma}
Let $u\in L^\infty(0,T; H^\gamma(\R^d))$ for some $ \gamma\in \Z^+ $. 
Then there exists a constant $C$ depending only on $\|u\|_{L^\infty(0,T; H^\gamma(\R^d))}$  such that 
\begin{align*}
\|u_L\|_{L^\infty(0,T; H^\gamma(\Omega_L))}\le C.
\end{align*}
\begin{proof}
By using the Kato--Ponce's inequality, it is straightforward to prove that, for $t\in[0,T]$, 
\begin{align*}
 \| J^\gamma_L (\chi_L u)\|_{L^2( \Omega_L)}
\le \| J^ \gamma \chi_L\|_{L^{\infty }(\R^d)} \|  u\|_{L^{2}(\R^d)}
	+ \| J^ \gamma  u\|_{L^2(\R^d)} \| \chi_L\|_{L^\infty(\R^d)}
\le C .
\end{align*}
\end{proof}
 \end{lemma}

\begin{lemma}\label{lem:r0}
Let $u\in L^\infty(0,T; H^\gamma(\R^d)\cap L^2(\R^d;|x|^{2\gamma}dx))$ for some $ \gamma\in \Z^+ $. 
Then there exists a constant $C$ depending only on 
$\|u\|_{L^\infty(0,T; H^\gamma(\R^d)\cap L^2(\R^d;|x|^{2\gamma}dx))}$ such that the remainder $R$ defined in \eqref{def:r} satisfies the following estimate: 
\begin{align*}
\|R\|_{L^\infty(0,T; L^2(\Omega_L))}\le C L^{-\gamma} .
\end{align*}
\end{lemma}
\begin{proof}
From the definition of $R$ in \eqref{def:r}, we take the $L^2$ norm to obtain
\begin{align}
\|R\|_{L^\infty(0,T;L^2(\Omega_L))}& \le
 \|\chi_L Vu- \chi_L^2 V  u\|_{L^\infty(0,T;L^2(\Omega_L))} \notag \\
& \quad+ \|2 \nabla \chi_L  \nabla u\|_{L^\infty(0,T;L^2(\Omega_L))}
  + \|\Delta\chi_L u\|_{L^\infty(0,T;L^2(\Omega_L))} \label{equ:b}.
\end{align}
Since $\chi_L = 1$ in $ \Omega_{L/2}$ and $|\chi_L |\le 1$, it follows that for any function $f\in L^2(\R^d;|x|^{2\gamma} dx)$ the following inequality holds: 
\begin{align} \label{est:f}
 \| f - \chi_L  f\|_{L^2(\R^d)} = \| f - \chi_L f\|_{L^2( {\Omega_{L/2}^c})}
 \lesssim \|f\|_{L^2( {\Omega_{L/2}^c})}
 \lesssim  L^{- \gamma} \| |x|^ \gamma f\|_{L^2(\R^d)}.
\end{align}
By applying \eqref{est:f}, the Bernstein's inequality and Lemma~\ref{lem:lwp}, we obtain the following estimate for the first term on the right-hand side of \eqref{equ:b}: 
\begin{align*}
\|\chi_L Vu- \chi_L^2 V  u\|_{L^\infty(0,T;L^2(\Omega_L))}  
  &\lesssim \|\chi_L V\|_{L^\infty( \Omega_L)} \| u-\chi_L u\|_{L^\infty L^2 ([0,T]\times \R^d)} \\
  &\lesssim L^{- \gamma} \| |x|^ \gamma u\|_{L^\infty L^2([0,T]\times \R^d)} . 
\end{align*}
Since $\chi_L$ is a constant in $ \Omega_{L/2}$ and $ \Omega_{L}^c$, it follows that $\textrm{supp}(\nabla\chi_L)$ is contained in $\Omega_{L/2}^c$.
As a result, the second and third terms on the right-hand side of \eqref{equ:b} can be estimated as follows:  
\begin{align*}
&\|2 \nabla \chi_L  \nabla u\|_{L^\infty(0,T;L^2(\Omega_L))}
  + \|\Delta\chi_L u\|_{L^\infty(0,T;L^2(\Omega_L))} \\
& \lesssim \| \nabla \chi_L\|_{L^\infty (\R^d)} 
\big(\| |\nabla| u\|_{L^\infty L^2([0,T]\times \Omega_{L/2}^c)} \big) 
	+\| \Delta\chi_L\|_{L^\infty( \R^d)} 
	\big(	\| u\|_{L^\infty L^2([0,T]\times  \Omega_{L/2}^c)}\big)\\
&\lesssim L^{-1} \big( L^{1-\gamma}  \| |x|^{\gamma-1} |\nabla| u\|_{L^\infty L^2([0,T]\times \R^d)}\big) +L^{-2- \gamma} \| |x|^ \gamma u\|_{L^\infty L^2([0,T]\times \R^d)}\\
&\lesssim L^{-\gamma},
\end{align*}
where the boundedness of $ \| |x|^{\gamma-1} |\nabla| u\|_{L^2(\R^d)}$ follows from Lemma 4 in \cite{NP2009CPDE}, i.e.,
\begin{align*}
\| |x|^{ \gamma-1} |\nabla| u\|_{L^2(\R^d)}
\le \| |x|^{\gamma} u \|_{L^2(\R^d)}^{ (\gamma-1)/\gamma} 
\| |\nabla|^\gamma u\|_{L^2(\R^d)}^{1/\gamma} \le  C\| u \|_{L^\infty(0,T;H^\gamma(\R^d)\cap L^2(\R^d;|x|^{2\gamma}dx))} .
\end{align*}
Substituting these estimates into \eqref{equ:b}, we obtain the result of Lemma \ref{lem:r0}. 
\end{proof}

\begin{lemma}\label{lem:ri}
Let $u\in L^\infty(0,T; H^\gamma(\R^d)\cap L^2(\R^d;|x|^{2\gamma}dx))$ for some $ \gamma\in \Z^+ $. 
Then there exists a constant $C$, depending only on 
$\|u\|_{L^\infty(0,T; H^\gamma(\R^d)\cap L^2(\R^d;|x|^{2\gamma}dx))}$,  such that the remainders $R_i^n$, $i=1,2,3$, defined in Section \ref{section:remainders} satisfy the following estimates: 

\begin{itemize}
\item[(i)] If $ \gamma=1$ then 
\begin{align*}
\|R_1^n + R_2^n + R_3^n\|_{L^2( \Omega_L)}
    \le C\tau \big( L^{-1} + (NL^{-1})^{- 1} +\tau N L^{-1} +\tau\big).
\end{align*}
\item[(ii)] If $ \gamma\ge2$ then 
\begin{align*}
\|R_1^n + R_2^n + R_3^n\|_{L^2(\Omega_L)}
    \le C\tau \big( L^{-\gamma} + (NL^{-1})^{- \gamma} +\tau \big).
\end{align*}
\end{itemize}
\end{lemma}

\begin{proof}
The definition of $R_1^n$ in Section \ref{section:remainders} and the result in Lemma \ref{lem:r0} imply that 
\begin{align*}
 \|R_1^n\|_{L^2(\Omega_L)}    \le \tau  \|R\|_{L^\infty(0,T;L^2(\Omega_L))}
   \le C\tau L^{-\gamma}.
\end{align*}
By the definition of $R_2^n$ in Section \ref{section:remainders}, we have
\begin{align*}
\|R_2^n\|_{L^2(\Omega_L)} 
  &  \le \tau \|\chi_L V u_L(t_n+s) - \tPN (I_{L,N}(\chi_L V) \tPN u_L(t_n+s))\|_{L^\infty(0,\tau;L^2(\Omega_L))}\\
  &\le \tau \| (1- \tPN) (I_{L,N}(\chi_L V) \tPN u_L(t_n+s))\|_{L^\infty(0,\tau;L^2(\Omega_L))}\\
  &\quad+ \tau \|  (\chi_L V - I_{L,N} (\chi_L V)) u_L(t_n+s)\|_{L^\infty(0,\tau;L^2(\Omega_L))}\\
 &\quad   + \tau \|I_{L,N}(\chi_L V)  (1-\tPN)u_L(t_n+s)\|_{L^\infty(0,\tau;L^2(\Omega_L))}.
\end{align*}
By Lemma \ref{lem:Bernstein} and considering that $V$ is smooth and has compact support, we obtain
\begin{align*}
\|R_2^n\|_{L^2(\Omega_L)} 
     \le C\tau (NL^{-1})^{- \gamma} \|u_L\|_{L^\infty H^\gamma ([0,T]\times \Omega_L)}.
\end{align*}
By the definition of $R_3^n$ in Section \ref{section:remainders} and the boundedness of $\|I_{L,N}(\chi_L V)\|_{L^\infty(\Omega_L)}$, we obtain
\begin{align*}
\|R_3^n\|_{L^2(\Omega_L)} 
   & \le C\tau \|\tPN \big(u_L(t_n+s)-u_L(t_n)\big)\|_{L^\infty(0,T;L^2(\Omega_L))}
\end{align*}
Then, using the expression of $u_L(t_n+s)-u_L(t_n)$, which is similar as \eqref{equ:solu}, we have
\begin{align*}
\|R_3^n\|_{L^2(\Omega_L)} 
   &\le C \tau \| \tPN (\tstd{is}-1) u_L(t_n) \|_{L^\infty(0,T;L^2(\Omega_L))}  \\
    &\quad   + C\tau \Big\|\int_{t_n}^{t_{n}+s} \tstd{i(t_n+s-s')} \big[\chi_L V u_L(s') + R(s')\big] ds' \Big\|_{L^\infty(0,T;L^2(\Omega_L))} .
\end{align*}
In the case $ \gamma=1$, combining  Lemma  \ref{lem:r0} with the inequality above, and utilizing Bernstein's inequality on $ \Omega_L$ (i.e., Lemma \ref{lem:Bernstein}), we obtain
\begin{eqnarray*}
  &&\|R_3^n\|_{L^2(\Omega_L)} \\
  & \le& C \tau^2 \big(\| u_L\|_{L^\infty(0,T;H^2(\Omega_L))} +  \| u_L\|_{L^\infty(0,T;L^2(\Omega_L))} + \| R \|_{L^\infty(0,T;L^2(\Omega_L))}\big)\\
  & \le& C \tau^2 \big( (1+N L^{-1})\| u_L\|_{L^\infty(0,T;H^1(\Omega_L))}  + \| R \|_{L^\infty(0,T;L^2(\Omega_L))}\big)\\
  &\le&  C \tau^2 \big( 1+ N L^{-1} + L^{-1}  \big)
  \le C \tau^2 (1+N L^{-1}).
\end{eqnarray*}
In the case $ \gamma\ge2$, the same argument leads to the following result:
\begin{eqnarray*}
  &&\|R_3^n\|_{L^2(\Omega_L)} \\
  & \le& C \tau^2 \big( \| u_L\|_{L^\infty(0,T;H^2(\Omega_L))} +  \| u_L\|_{L^\infty(0,T;L^2(\Omega_L))} + \| R\|_{L^\infty(0,T;L^2(\Omega_L))}\big) 
  \le C \tau^2 .
\end{eqnarray*}
This completes the proof of Lemma \ref{lem:ri}. 
\end{proof}

\subsection{Error estimates}
We consider the difference between $u_L(t_{n+1})$ and the numerical solution $u_L^{n+1}$, i.e., 
\begin{align*}
&\| u_L(t_{n+1}) - u_L^{n+1}\|_{L^2( \Omega_L)} \\
  & \le \| u_L(t_{n}) - u_L^n\|_{L^2( \Omega_L)}  
+ \tau \|  \varphi_1(i \tau \Delta) \tPN \big( I_{L,N}(\chi_L V)\cdot  (u_L(t_n)-u_L^n)\big)\|_{L^2( \Omega_L)}\\
  &\quad  +\|R_1^n+R_2^n+R_3^n\|_{L^2( \Omega_L)}\\
 &\le (1+ C\tau) \| u_L(t_{n}) - u_L^n\|_{L^2( \Omega_L)}  
 +\|R_1^n+R_2^n+R_3^n\|_{L^2( \Omega_L)}.
\end{align*}
By the first result of Lemma \ref{lem:Bernstein}, we have 
\begin{align*}
\|u_L(0)-u_L^0\|_{L^2( \Omega_L)} 
&= \| \chi_Lu_0- \tPN (\chi_L u_0)\|_{L^2( \Omega_L)} 
\lesssim (N L^{-1})^{- \gamma} \| \chi_L u_0 \|_{H^\gamma(\Omega_L)} 
\end{align*}
where $\| \chi_L u_0 \|_{H^\gamma(\Omega_L)}  \lesssim \|  u_0 \|_{H^\gamma(\Omega_L)} $. 
In the case $ \gamma=1$, we obtain by Gronwall's inequality and Lemma \ref{lem:ri} that
\begin{align}\label{est:gam}
\| u_L(t_{n+1}) - u_L^{n+1}\|_{L^2( \Omega_L)} 
 &\le 
 \left\{
 \begin{aligned}
 &C\big( L^{-1}  + (NL^{-1})^{- 1} +\tau N L^{-1}+\tau \big) &&\mbox{for}\,\,\,\gamma=1 , \\
 &C\big( L^{-\gamma}  + (NL^{-1})^{- \gamma} +\tau \big) &&\mbox{for}\,\,\,\gamma\ge 2 .
 \end{aligned}\right. 
\end{align}
Then 
\begin{align*}
\|u(t_n)- Eu_L^n\|_{L^2(\R^d)}
&\le\|u(t_n)- \chi_L u(t_n)\|_{L^2(\R^d)}
+\|u_L(t_n)- u_L^n\|_{L^2( \Omega_L)}\\
&\lesssim 
\|(\chi_L - 1)u(t_n)\|_{L^2(\R^d)}
+\|u_L(t_n)- u_L^n\|_{L^2( \Omega_L)}\\
&\lesssim L^{- \gamma} \| |x|^ \gamma u(t_n)\|_{L^2(\R^d)} 
+\|u_L(t_n)- u_L^n\|_{L^2( \Omega_L)} \\
&\lesssim 
 \left\{
 \begin{aligned}
 &L^{-1}  + (NL^{-1})^{- 1} +\tau N L^{-1} +\tau&&\mbox{for}\,\,\,\gamma=1 , \\
 &L^{-\gamma}  + (NL^{-1})^{- \gamma} +\tau &&\mbox{for}\,\,\,\gamma\ge 2 ,
 \end{aligned}\right. 
\end{align*}
where the last inequality follows from \eqref{est:gam}. This completes the proof of Theorem \ref{theorem-lsp}.\hfill\endproof

\section{Numerical results}\label{section:numerical}
In this section, we present numerical examples to support the theoretical analysis on the convergence rates of the proposed method (as shown in Theorem \ref{theorem-lsp}), and to demonstrate the the performance of the method on quantum tunneling (Section \ref{section:quantum_tunneling}) and quantum scattering from lattices (Section \ref{section:quantum_scattering}). 

\subsection{The free Schr\"odinger equation}
We consider the linear Schr\"odinger equation \eqref{model-lsp} in one dimension with $V(x) = 0$ and an initial condition given by a Gaussian wave packet:
\begin{align}\label{equ:free}
i \partial_t u + \partial_{xx} u = 0 
\quad\mbox{with}\quad 
u(x,0) = e^{-x^2/9+i x}. 
\end{align}
In this scenario, the exact solution is known explicitly, i.e.,
\begin{align*}
u(x,t) = \frac{1}{ \sqrt{1+ \frac{4}{9} it}}
\exp \left( -\frac{it}{90}  \frac{(2t-x)^2}{90+40it} - \frac{x^2}{90}\right),
\end{align*}
where the modulus $|u|$ propagates with a positive velocity $v = 2$. 
\begin{figure}[htp!]
\centering
\includegraphics[width=0.45\textwidth]{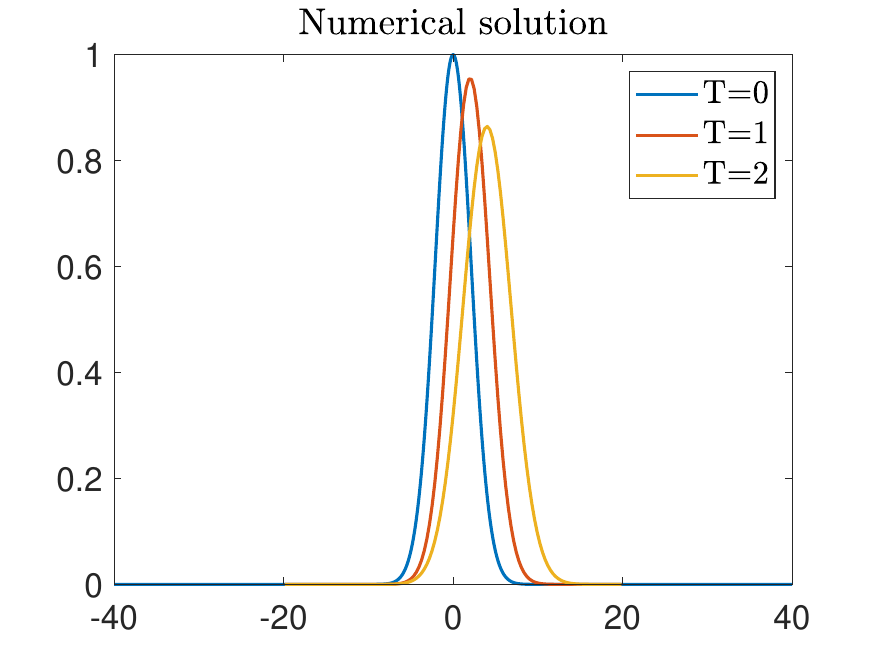}
\includegraphics[width=0.45\textwidth]{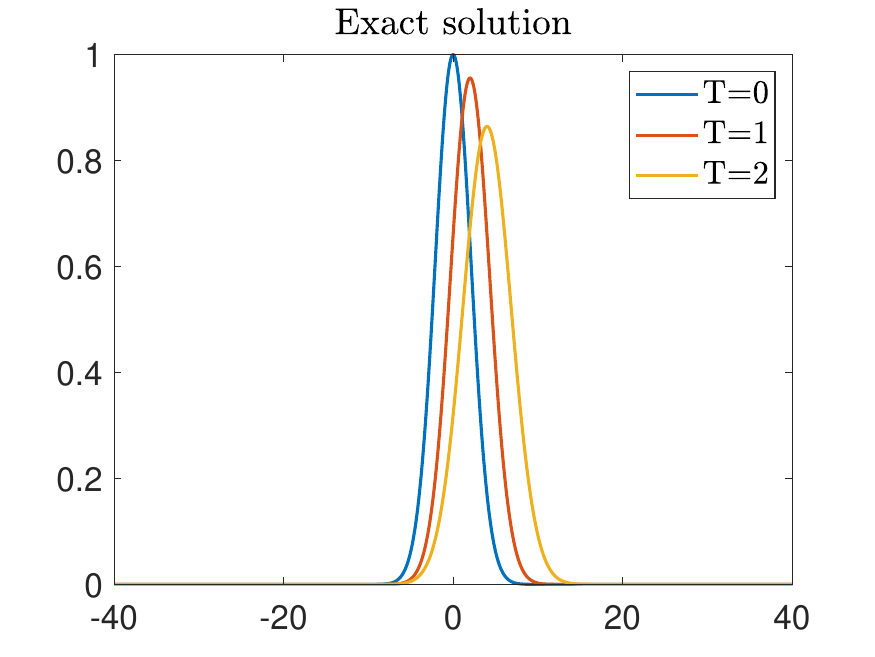}
\vspace{-10pt}
\caption{Modulus of the exact and numerical solutions (free Schr\"odinger equation).}
\label{NumerSol}
\end{figure}

This allows us to conveniently measure the error between the numerical solutions and the exact one. The modulus of the numerical and exact solutions at $T = 0, 1, 2$ are displayed in Figure \ref{NumerSol}, where the numerical solution is given by the proposed method with $L=40$ and $N=1600$. Clearly, the numerical solution produces the correct shape of the wave.

The absence of a potential means there is no temporal discretization error. Therefore, we only demonstrate the convergence of the spatial discretizations. We set $L = N^{1/2}$ and apply the proposed method with $N = 2^j$ for $j = 4, \ldots, 10$ (where $N$ denotes the degrees of freedom in the spatial discretizations). The numerical results in Figure \ref{CVSol1} show that the $L^2$ errors from the spatial discretizations converge faster than any algebraic power, which is consistent with the result proved in Theorem \ref{theorem-lsp}, given that the initial solution is smooth and no temporal error is involved.

\begin{figure}[htp!]
\centering
\includegraphics[width=0.45\textwidth]{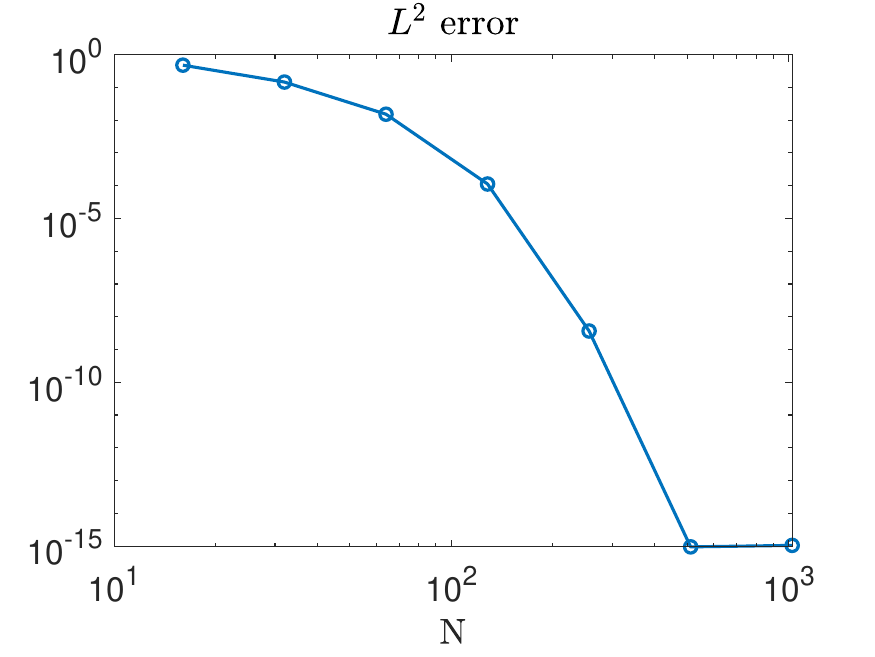}
\vspace{-10pt}
\caption{Errors of the numerical solutions at $T = 1$ (free Schr\"odinger equation).}
\label{CVSol1}
\end{figure}

Since the original problem \eqref{equ:free} is posed on the whole space, the exact solution will eventually propagate to the boundary of the predefined $L$. In practice, to prevent nonphysical artifacts caused by spatial truncation and periodic boundary conditions and to ensure that the solution remains away from the extended boundary, we can dynamically adjust the bounded window of computation to simulate wave propagation by monitoring the following values (this is beyond the scope of analysis in this paper):
$$
F(u_L^n) = \max\{|u_L^n(-L)|, |u_L^n(L)|\}. 
$$
If these values exceed a specified threshold $\epsilon$, it indicates that the solution has reached the boundary. In this case, we extend the window from $[-L, L]$ to $[-2L, 2L]$. Furthermore, the numerical solution $u_L^n$ at this point should be extended to $[-2L, 2L]$ using a zero extension operator $E$ to become the updated initial function. We refer to this process as domain extension and outline the procedure in Algorithm \ref{euclid}.

\begin{algorithm}
\caption{Extend computational domain for numerical solution $u_L^n$ with tolerance $ \epsilon$}
\label{euclid}
\begin{algorithmic}[1]
\STATE $e \gets F(u_L^n)$ 
\WHILE{$e \ge \epsilon$}
\STATE $L \gets 2L$
\STATE $N \gets 2N$
\STATE $u_L^n \gets P_{L,N} I_{L, 2N} Eu_L^n$
\STATE Continue computation of \eqref{scheme-p} with the updated $L$, $N$ and $u_L^n$.
\ENDWHILE
\end{algorithmic}
\end{algorithm}

\begin{figure}[htp!]
\centering
\includegraphics[width=0.45\textwidth]{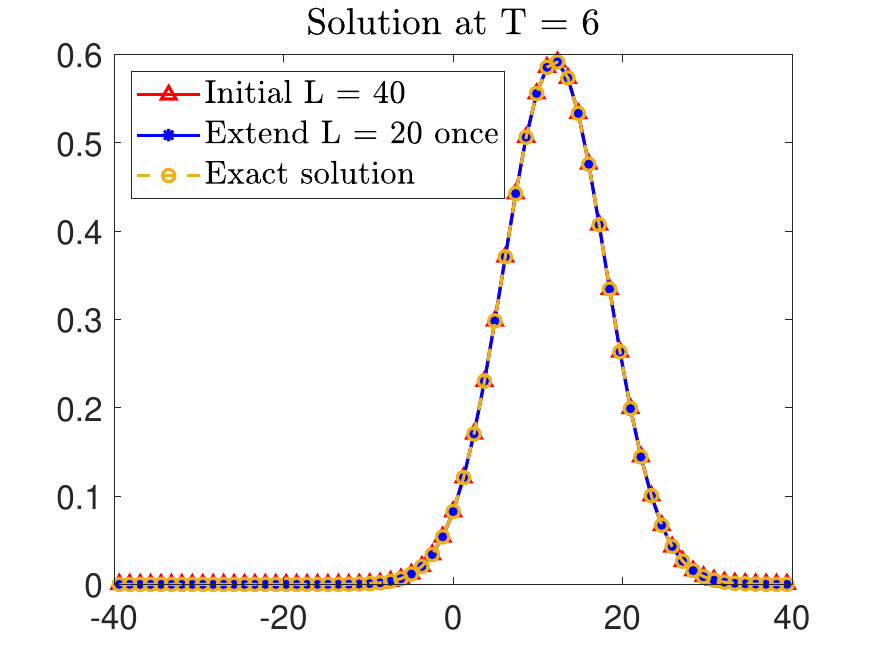}
\vspace{-10pt}
\caption{Comparison between exact and numerical solutions 
 (free Schr\"odinger equation). 
}
\label{ExtendSol}
\end{figure}

The effectiveness of the domain extension technique is demonstrated in Figure \ref{ExtendSol}. At $T = 6$, the numerical solution computed using a sufficiently large $L = 40$ does not require such domain extension, as the solution remains close to zero in the vicinity of the boundaries. However, when starting with $L = 20$, the numerical solution approaches the boundary's vicinity in the computation. Consequently, a domain extension is performed at $T_e = 3.2344$ according to Algorithm \ref{euclid}. In Figure \ref{ExtendSol}, the solutions from both scenarios align closely, highlighting the effectiveness of the domain extension procedure.

\subsection{Quantum tunneling in one dimension}\label{section:quantum_tunneling}

We consider an example of quantum tunneling modelled by equation \eqref{model-lsp} in one dimension with an external potential $V(x) = 200\,b(10x)$ having compact support, where 
\[
b(x) =
\begin{cases}
\exp\left(-\frac{1}{1-x^2}\right), & \text{if } |x|\leq 1,\\
0 ,& \text{if } |x| > 1.
\end{cases}
\] 
The following three types of initial conditions are considered.
\begin{itemize}
\item[(i)] Type I. Smooth initial function:
\begin{align*}
u_0(x) = \exp(-(x+5)^2 + 8i(x+5)).
\end{align*}
\item[(ii)] Type II. $H^1$ initial function:
\begin{align*}
u_0(x) = {|x+8|}^{0.51}\exp(-(x+8)^2)\exp(4i(x+8)).
\end{align*}
\item[(iii)] Type III. $H^2$ initial function:
\begin{align*}
u_0(x) = (x+8){|x+8|}^{0.51}\exp(-(x+8)^2)\exp(4i(x+8)).
\end{align*}
\end{itemize}

\begin{figure}[htp!]
\centering
\includegraphics[width=0.45\textwidth]{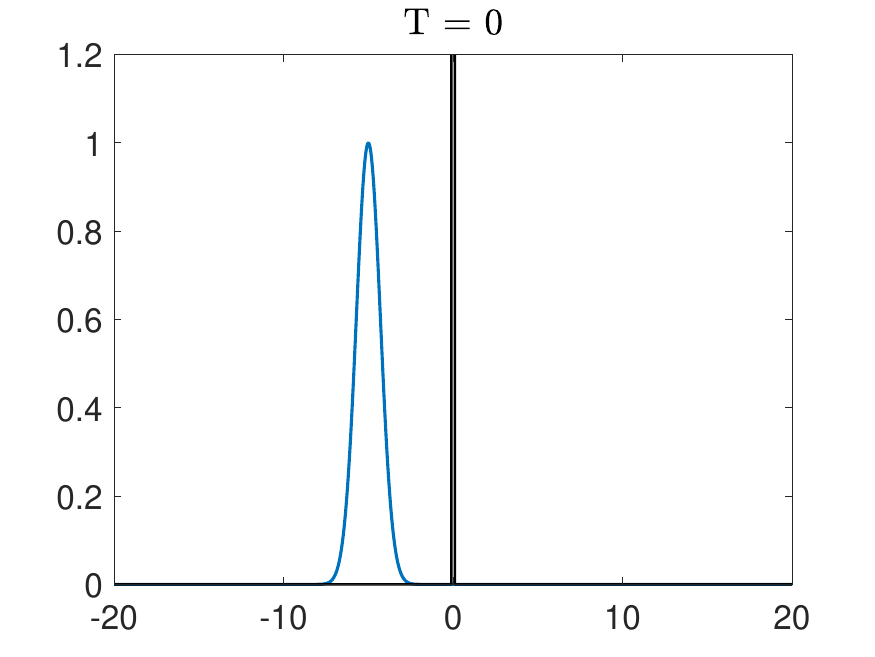}
\includegraphics[width=0.45\textwidth]{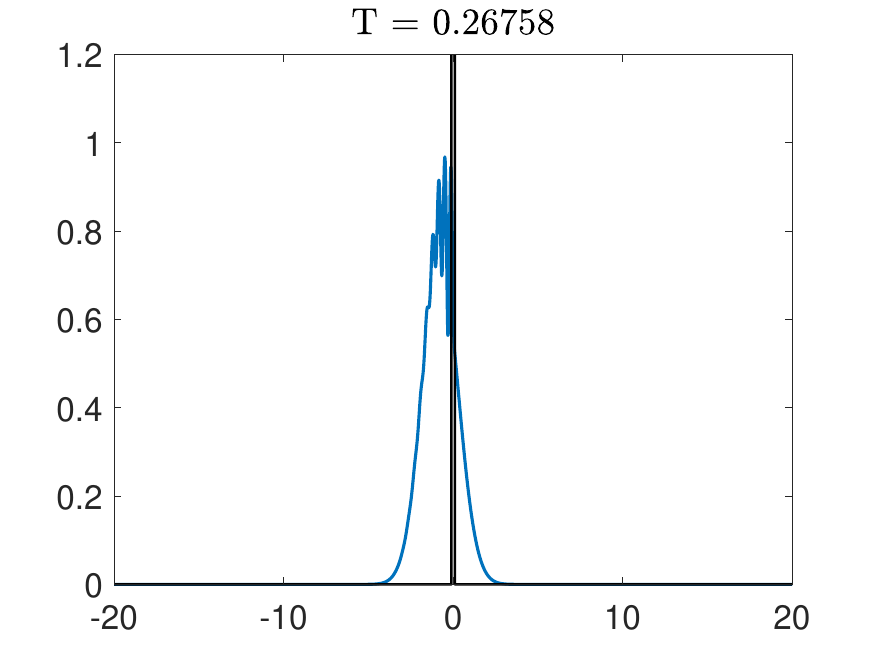}\\
\includegraphics[width=0.45\textwidth]{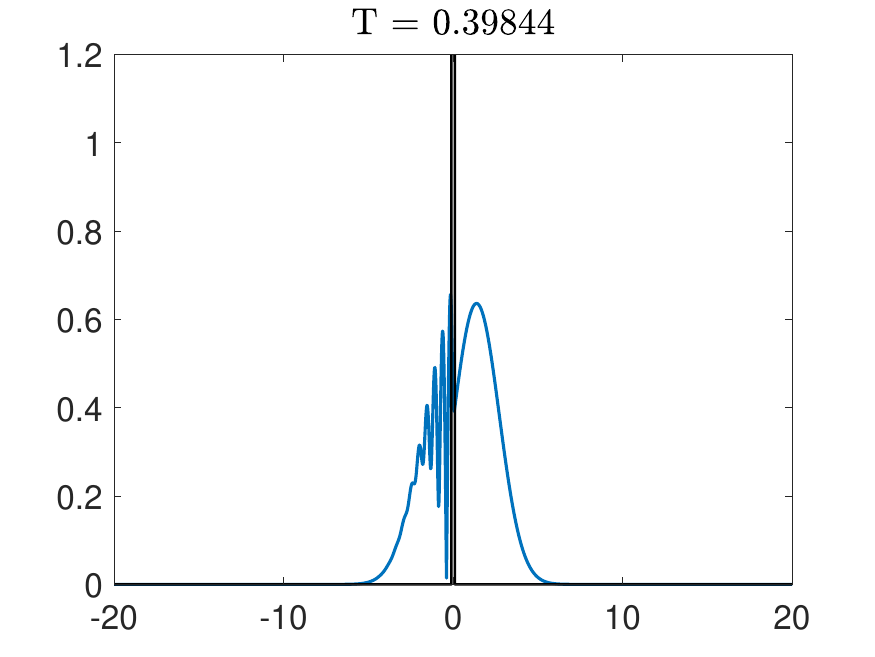}
\includegraphics[width=0.45\textwidth]{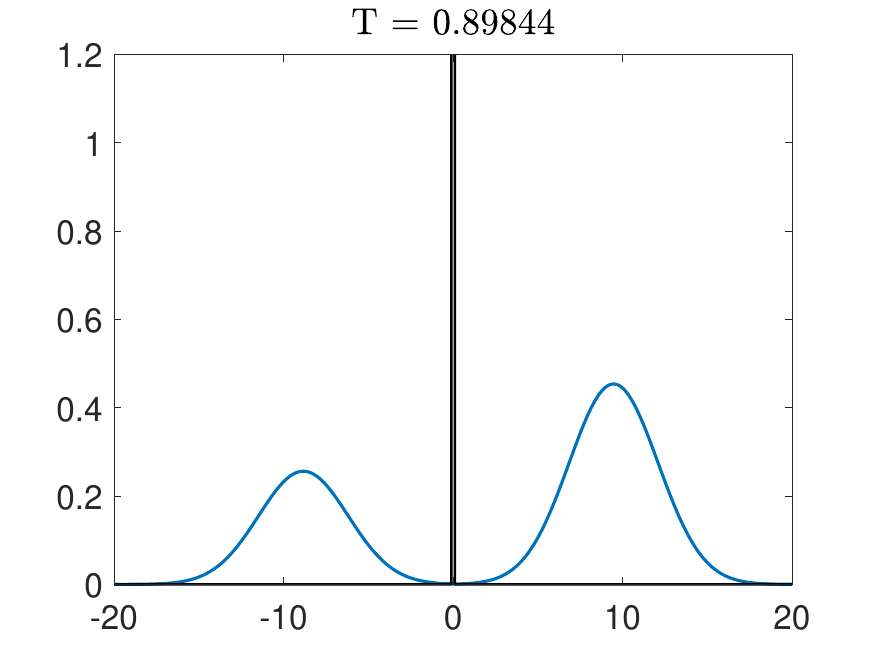}
\caption{
Numerical solution of $|u|$ with Type I initial function (quantum tunneling example).}
\label{DMpic1d}
\end{figure}

\begin{figure}[htp!]
\centering
\subfigure[Convergence in space]{\includegraphics[width=0.45\textwidth]{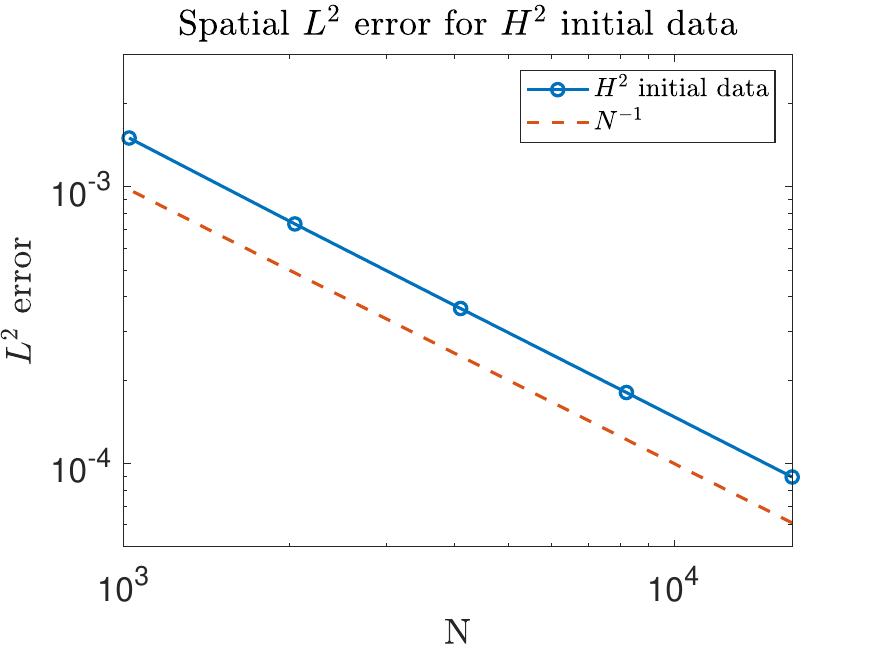}}
\subfigure[Convergence in time]{\includegraphics[width=0.45\textwidth]{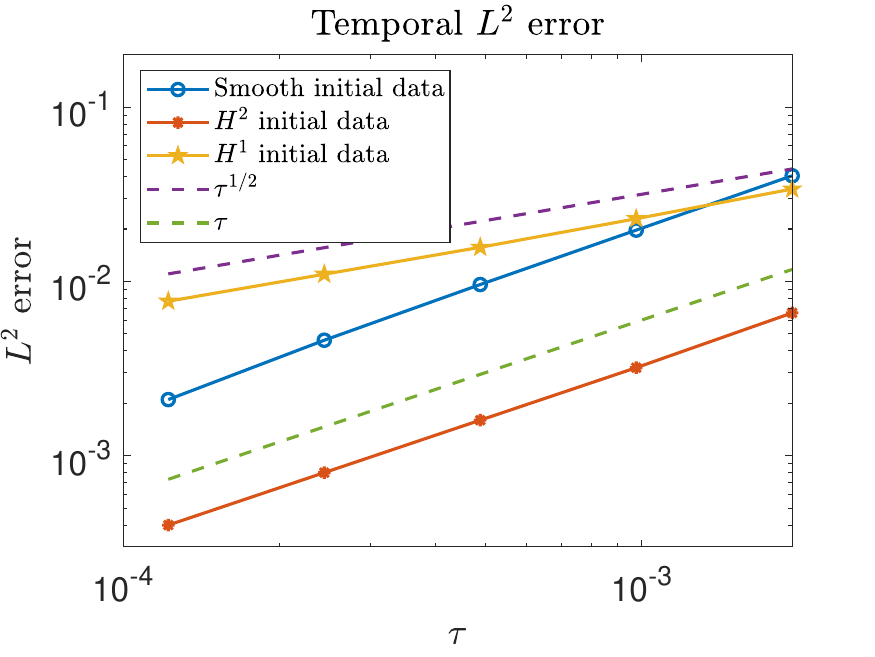}}
\caption{Convergence rate of numerical solution at $T = 1$ (quantum tunneling example).
}
\label{CVRate2}
\end{figure}

We begin by demonstrating the numerical solutions for the Type I smooth initial function. As shown in Figure \ref{DMpic1d}, numerical solution initially propagates to the right direction before interacting with the potential. Subsequently, part of the wave is reflected back while another part tunnels through the potential. Numerical solutions for the initial functions of Type II and Type III exhibit similar behavior. 

Next, we support the error estimates proved in Theorem \ref{theorem-lsp} through numerical experiments. Since the exact solution is unavailable, a reference solution at $T = 1$ is computed using $N = 2^{16}$, $\tau = 2^{-16}$, and $L = 2^8$. For the $H^2$ initial function of Type III, we select $N = 2^{10}, \ldots, 2^{14}$ and $L = N^{1/2}$ with a sufficiently small stepsize $\tau = 2^{-16}$ such that the errors from time discretization can be ignored. The numerical results in Figure \ref{CVRate2} (a) indicate that first-order convergence in space discretization is achieved for the $H^2$ initial function. 

According to Theorem~\ref{theorem-lsp}, half-order convergence should be achieved for the $H^1$ initial function (Type II) by choosing $L = N^{1/2}$ and $N = \tau^{-1}$, and first-order convergence with respect to $\tau$ should be achieved for the $H^2$ and smooth initial functions (Type I and Type III) by choosing sufficiently large $N= 2^{16}$ and $L=N^{1/2}$. 
These results are also supported by the numerical results shown in Figure \ref{CVRate2} (b). 

For long time simulation, we apply the domain extension technique outlined in Algorithm \ref{euclid} (this is beyond the scope of analysis in this paper). For the Type I smooth initial function, $L=20$ is chosen at $t=0$ and domain extension is performed at $t = 0.9452$ and $t = 1.7411$ according to Algorithm \ref{euclid}. The numerical result is very close to that computed directly by choosing $L = 80$ at $t=0$, as shown in Figure \ref{potentialcompare} (a). For the Type III $H^2$ initial function, $L=40$ is chosen at $t=0$ and domain extension is performed at $t = 1.3008$. Again, the numerical result is very close to that computed directly by choosing $L = 80$ at $t=0$, as shown in Figure \ref{potentialcompare} (b).

\begin{figure}[htp!]
\centering
\subfigure[Type I initial function]{\includegraphics[width=0.45\textwidth]{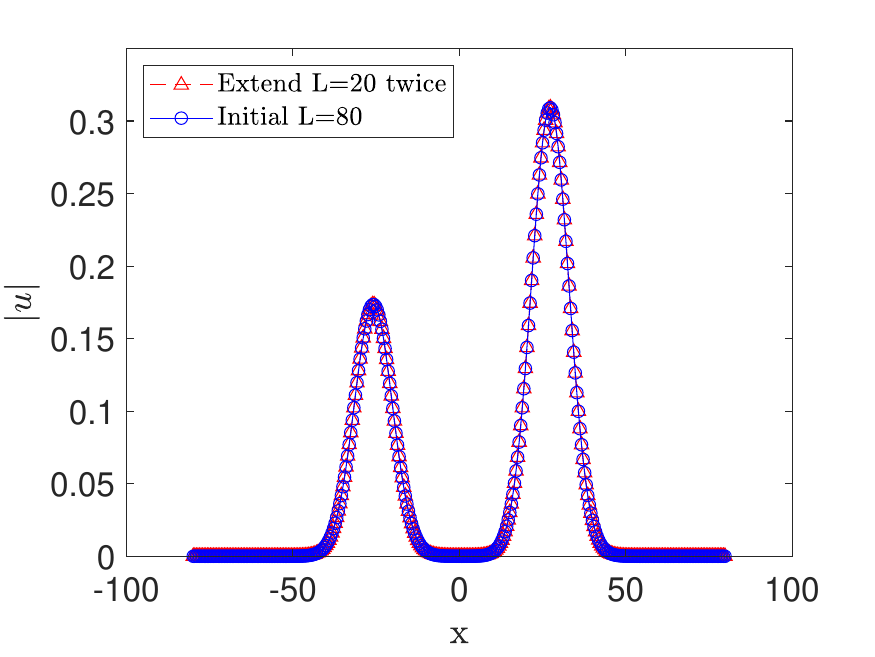}}
\subfigure[Type III initial function]{\includegraphics[width=0.45\textwidth]{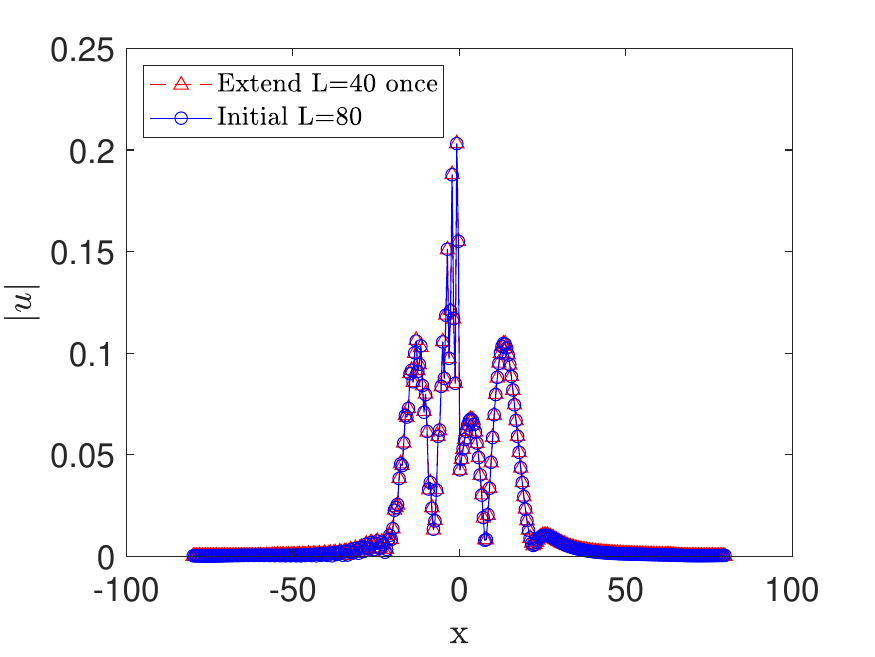}}
\caption{
Comparison between numerical solutions (quantum tunneling example).
}
\label{potentialcompare}
\end{figure}

%
%

\subsection{Quantum scattering in two dimensions}\label{section:quantum_scattering}

We consider an example of quantum scattering from crystal lattices modelled by equation \eqref{model-lsp} in two dimensions with the following external potential: 
\begin{equation*}
V(x,y) = 10\sum_{i = -1}^1\sum_{j = -5}^5 b(4((x-i)^2+(y-6j/5)^2)),
\end{equation*}
where 
\[
b(x,y) =
\begin{cases}
\exp\left(-\frac{1}{1-x^2-y^2}\right), & \text{if } x^2+y^2\leq 1,\\
0 ,& \text{if } x^2+y^2 > 1.
\end{cases}
\]
The following two types of initial conditions are considered.
\begin{itemize}
\item[(i)] Type I. Smooth initial function:
\begin{align*}
u_0(x,y) = e^{-r^2} e^{4i(x+2)},\quad r = ((x+2)^2+y^2)^{1/2}.
\end{align*}
\item[(ii)] Type II. $H^2$ initial function:
\begin{align*}
u_0(x,y) = r^{1.02} e^{-r^2} e^{4i(x+2)}.
\end{align*}
\end{itemize}
The potential and Type I initial function are shown in Figure \ref{potential}. 

\begin{figure}[htp!]
\centering
\includegraphics[width=0.45\textwidth]{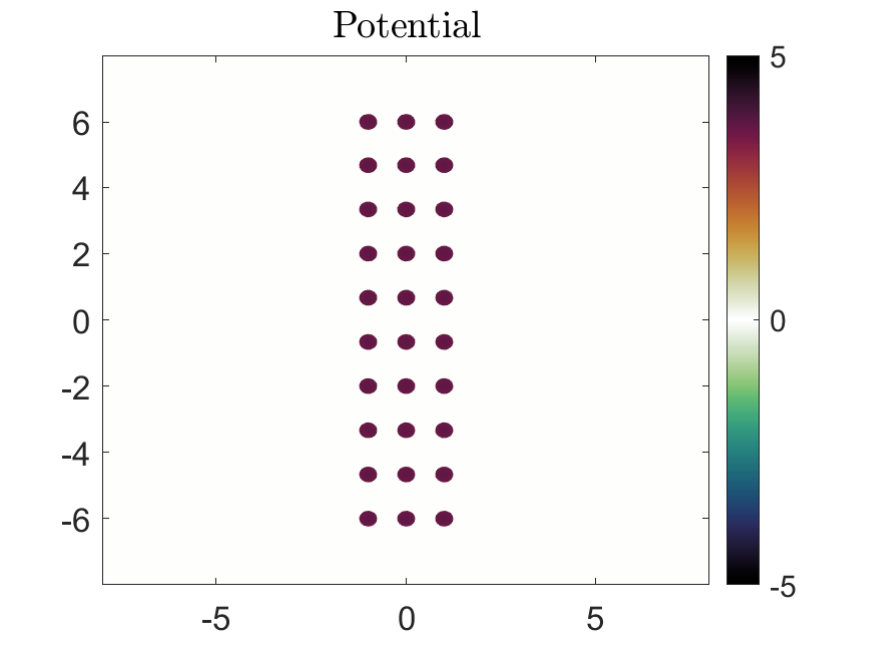}
\includegraphics[width=0.45\textwidth]{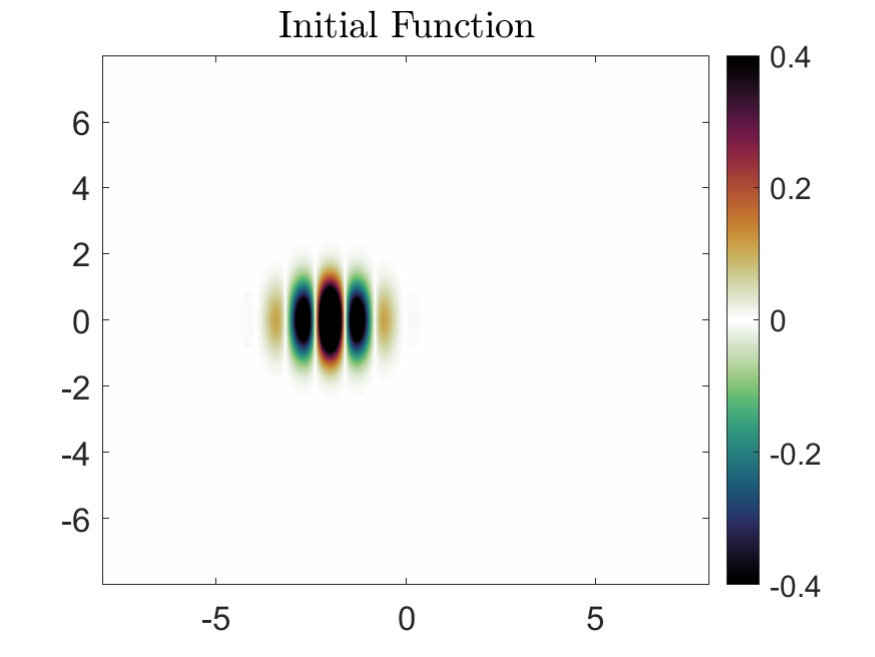}
\caption{Potential and Type I initial function  (quantum scattering example).}
\label{potential}
\end{figure}

\begin{figure}[htp!]
\subfigure[Errors of spatial discretizations]{\includegraphics[width=0.45\textwidth]{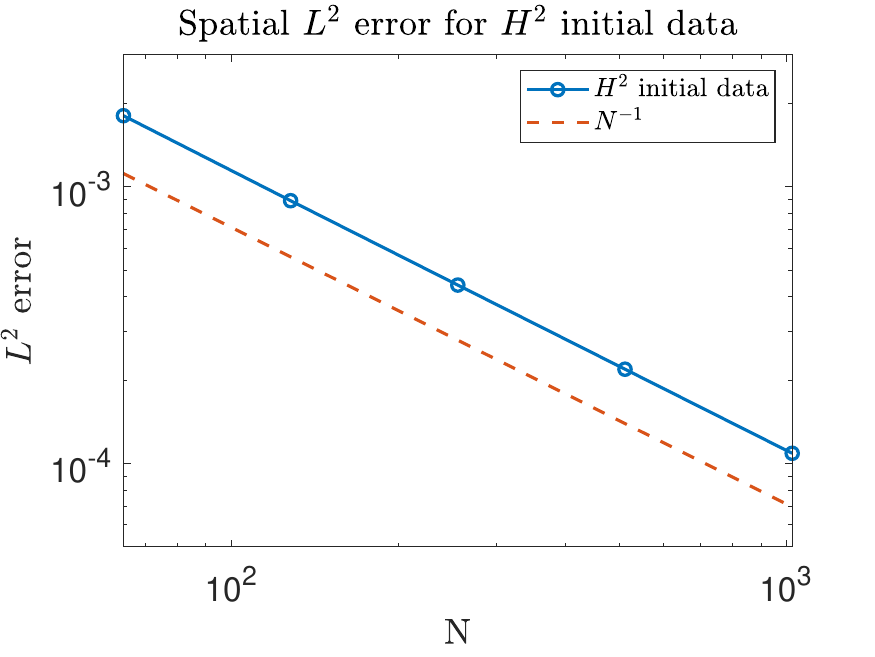}}
\subfigure[Errors of temporal discretizations]{\includegraphics[width=0.45\textwidth]{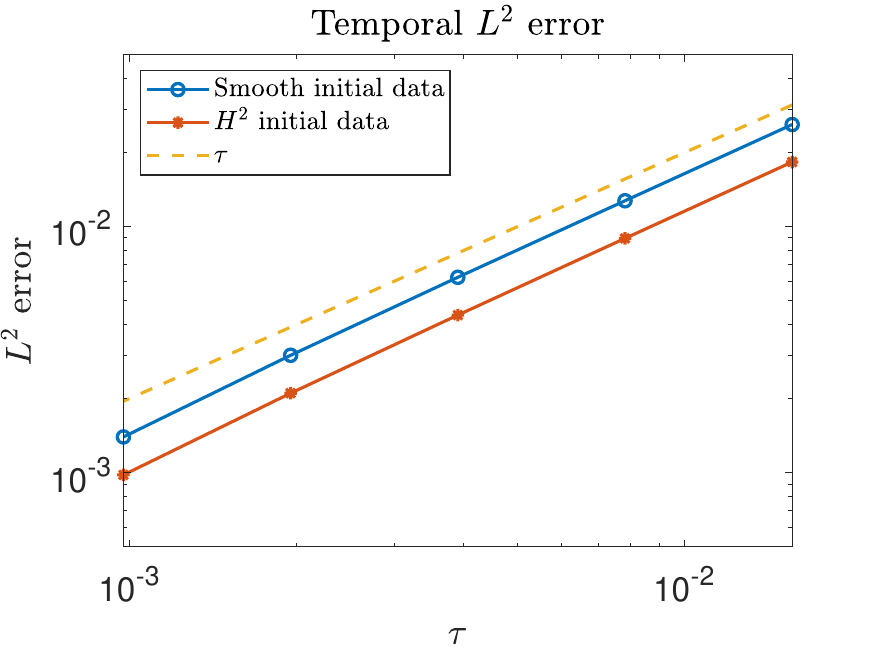}}
\caption{
Errors and convergence rates (quantum scattering example).
}
\label{Fig8}
\end{figure}

The convergence of numerical solutions is tested by comparing the numerical solutions with a reference solution at $T = 1/4$, computed by using $N = 2^{13}$ and $\tau = 2^{-13}$ for both types of initial functions. For the Type II $H^2$ initial function, the spatial discretization errors of the numerical solutions with $N = 2^j$ with $j = 6, \ldots, 10$ and a fixed $\tau = 2^{-13}$ are shown in Figure \ref{Fig8} (a), which demonstrates first-order convergence with respect to $N$. Moreover, the temporal discretization errors for both Type I and Type II initial functions are shown in Figure \ref{Fig8} (b), which demonstrates first-order convergence with respect to $\tau$. These numerical results are consistent with the theoretical results proved in Theorem~\ref{theorem-lsp}.

The effectiveness of domain extension in computation is tested for the Type I initial function and shown in Figure \ref{Fig9}, which shows that the solution obtained from a small initial domain $L = 10$, with domain extension performed once at $T = 0.5664$, closely matches the solution from a larger initial domain with $L=20$ from the beginning. Similar phenomena can also be observed in the case of the Type II initial function and omitted here. 

\begin{figure}[h!]
\includegraphics[width=0.45\textwidth]{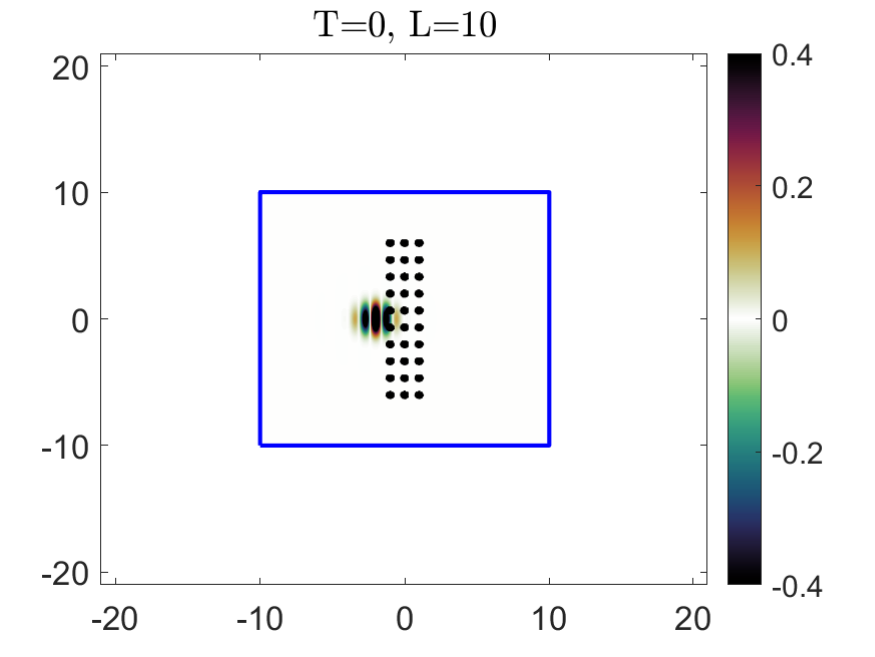}
\includegraphics[width=0.45\textwidth]{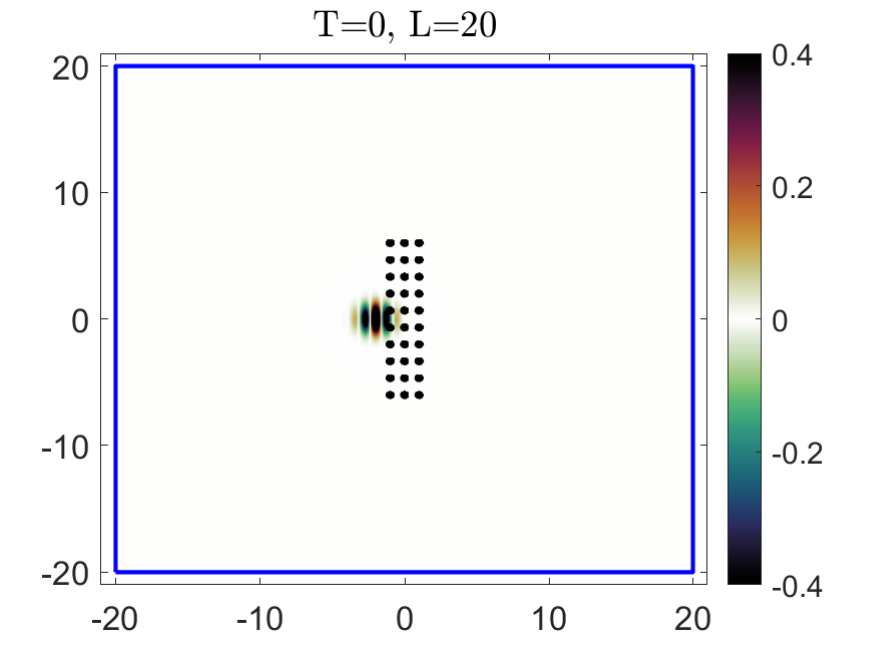}\\
\includegraphics[width=0.45\textwidth]{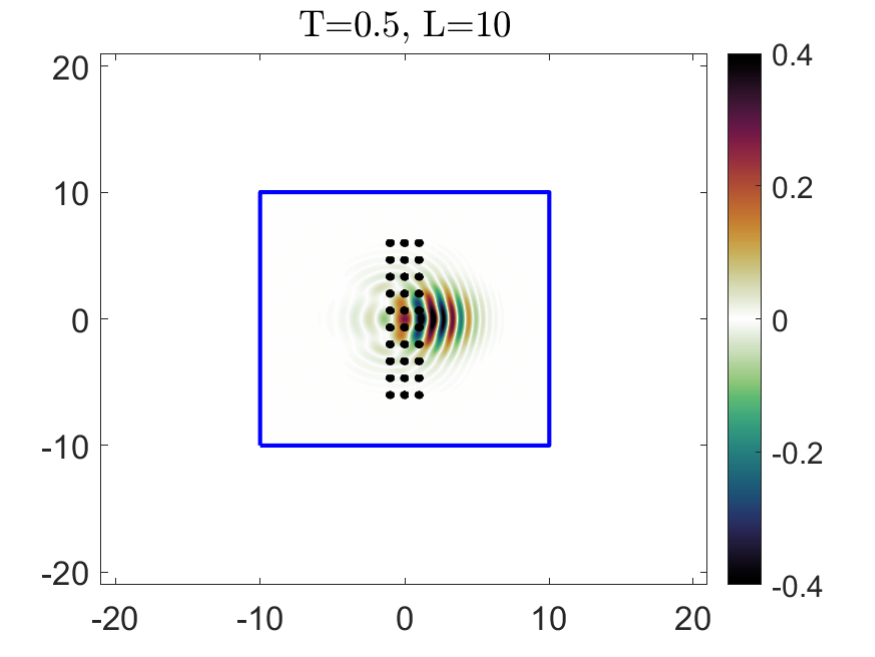}
\includegraphics[width=0.45\textwidth]{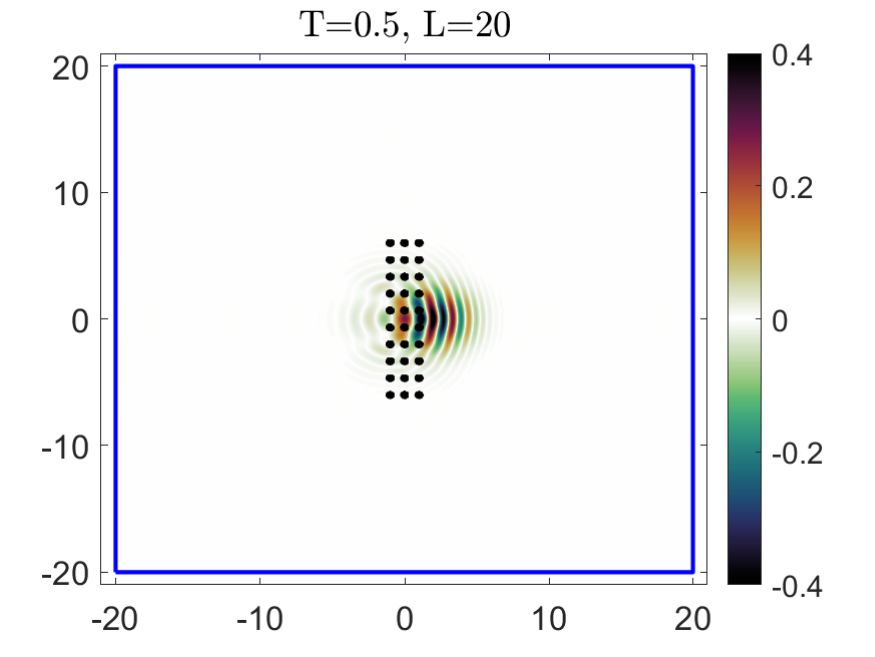}\\
\includegraphics[width=0.45\textwidth]{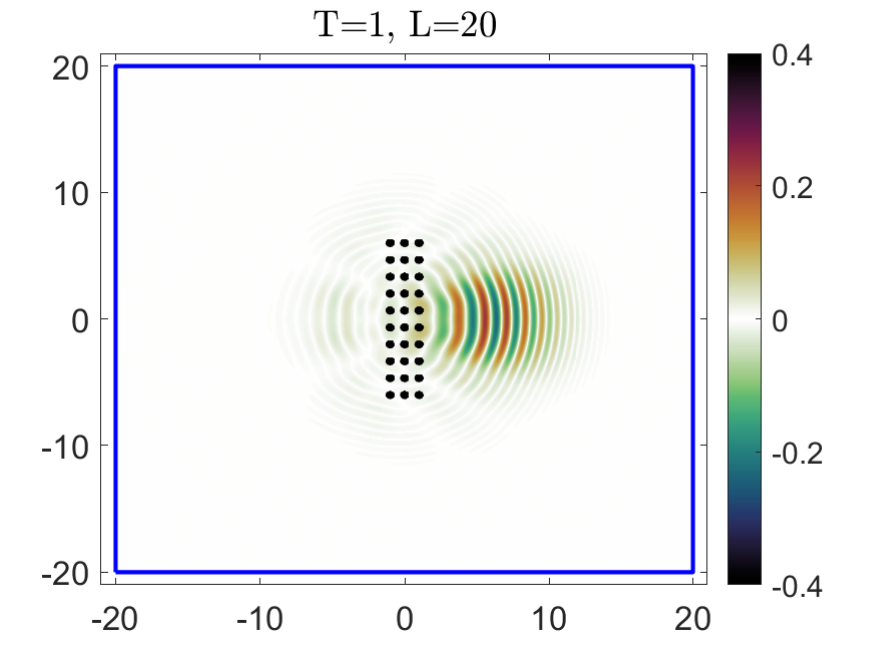}
\includegraphics[width=0.45\textwidth]{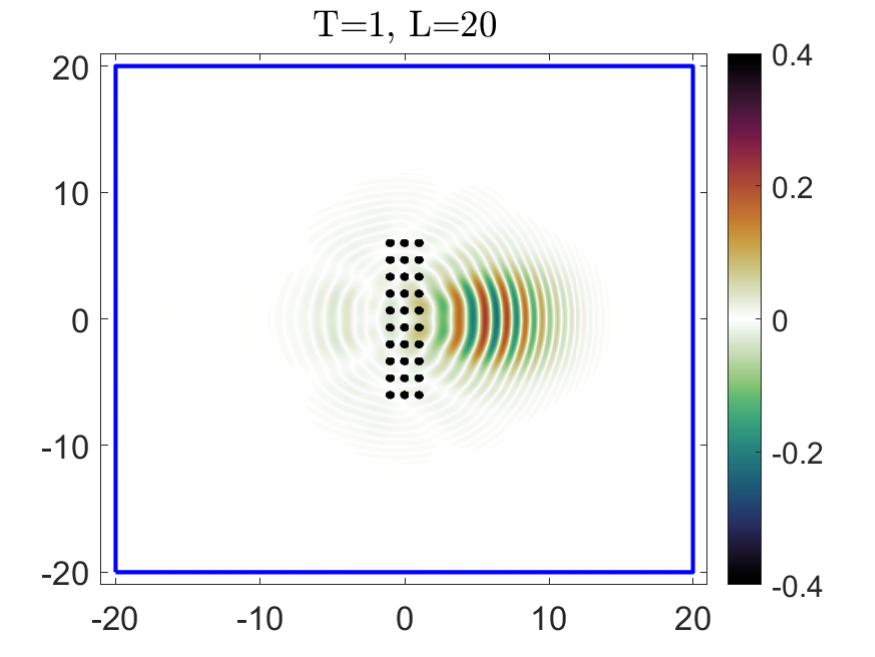}
\caption{
Effectiveness of domain extension (quantum scattering example).\newline
Figures on the left column: $L=10$ at $t=0$, and domain is extended at $T = 0.5664$.\newline
Figures on the right column: $L=20$ at $t=0$, without domain extension. 
}
%
\label{Fig9}
\end{figure}

\pagebreak
\section{Conclusion}\label{section:conclusion}
We have proposed a novel framework for solving the linear Schr\"odinger equation with an external potential in $\mathbb{R}^d$, by employing a smooth cut-off function to truncate the equation from the whole space to a scaled torus and solving the equation on the scaled torus by a first-order exponential integrator in time and Fourier spectral method in space. We have proved that the numerical solution on the scaled torus, with zero extension to $\R^d$, converges to the solution of the original problem in the whole space with first-order  convergence in time and $\gamma/2$-order convergence in space for initial data in $H^\gamma(\mathbb{R}^d) \cap L^2(\R^d;|x|^{2\gamma} dx)$ with $\gamma \geq 2$. 
In the case where $\gamma = 1$, we have proved that the numerical scheme has half-order convergence under an additional CFL condition $\tau=O(N^{-1})$. For practical computation, beyond the theoretical analysis in this paper, we have proposed an algorithm for dynamically adjusting the truncated domain, allowing for continued computation beyond the initial truncated domain. The framework developed in this paper is expected to serve as a fundamental approach to solving dispersive and wave equations in the whole space and analyzing the errors in approximating solutions in the whole space by numerical solutions on a scaled torus (with zero extension to the whole space).


\appendix
\renewcommand{\thesection}{Appendix}
\section{Proof of Lemma \ref{lem:lwp}} \label{appendix}
\renewcommand{\theequation}{A.\arabic{equation}}

For $u_0\in H^\gamma(\R^d)\cap L^2(\R^d;|x|^{2\gamma}dx)$, with $\gamma \in \Z^+$, it is well known that equation \eqref{model-lsp} has a unique solution $u\in C(\R; H^\gamma(\R^d)) $ satisfying 
$$
\|u\|_{L^\infty(0,T;H^\gamma(\R^d))}\le C_T \| u_0\|_{H^\gamma(\R^d)} 
$$ 
for any $T>0$. 
It remains to prove that $u\in C(\R; L^2(\R^d;|x|^{2 \gamma} d x))$. To this end, we adopt the idea in \cite{NP2009CPDE} for the nonlinear Schr\"odinger equation by considering the partial differential operator \begin{align}\label{equ:g} 
\Gamma_j  := (x_j + 2it\partial_{x_j})  
                          &= e^{it \Delta} x_j e^{-it \Delta} , \quad j=1,\cdots, d , 
\end{align}
where the last equality can be verified by applying the Fourier transform to $e^{it \Delta} x_j e^{-it \Delta}$. 
By using the multi-index notation $\Gamma^ \beta(t) := (x + 2it\nabla )^ \beta $ for $\beta\in (\Z_{\geq 0})^d , $ 
we can derive from \eqref{equ:g} that
\begin{align} \label{equ:com}
\Gamma^ \beta(t) \tstd{it} \cdot =\tstd{it}  (x^ \beta \ \cdot).
\end{align}

For any fix $\M>0$ and $T>0$, which will be chosen later, we denote by $E$ the set of functions $u$ such that $\|u\|_E \le {\M}$, with  
$$
\|u\|_E:= 
\sum_{ | \beta| \leq \gamma}  \|\Gamma^ \beta(t)  u\|_{L^\infty  L^2([0,T]\times \R^d)}
+\| u \|_{L^\infty  H^ \gamma([0,T]\times \R^d)} , 
$$
and let $ (E, \rho) $ be the complete metric space with distance $ \rho(u,v) = \|u-v\|_E$. We aim to prove that, by choosing sufficiently large $M$ and sufficiently small $T$, we can define a contraction map $\Phi:E\rightarrow E$ as follows: 
$$
\Phi (u) : =\tstd{it} u_0 -i\int_0^t \tstd{i(t-s)} (V(x)u(x,s))\,ds .
$$ 
This would imply the existence and uniqueness of solutions in $E$. 

We first prove that $\Phi(u)\in E$ for any $u\in E$. This can be proved by utilizing \eqref{equ:g} and \eqref{equ:com}, which imply that
 \begin{align*}
  \Gamma^ \beta(t) \Phi(u)& = \tstd{it}  (x^ \beta u_0) 
         - i\int_0^t \tstd{it} x^ \beta \tstd{-is} (V(x)u(x,s))\,ds\\
       &= \tstd{it}  (x^ \beta u_0) 
         - i\int_0^t \tstd{i(t-s)} \tstd{is} x^ \beta \tstd{-is} (V(x)u(x,s))\,ds\\
       &= \tstd{it}  (x^ \beta u_0) 
         - i\int_0^t \tstd{i(t-s)} \Gamma^ \beta(s)(V(x)u(x,s))\,ds.
      \end{align*}
Considering the smoothness and compact support of potential $V(x)$, along with the fact that $u\in E$, we derive that
     \begin{align}\label{est:non}
    \| \Phi(u) \|_E
        & \le  \sum_{ |\beta|\le \gamma} \|x^ \beta u_0\|_{ L^2 (\R^d)} + \|u_0\|_{H^ \gamma(\R^d)}
          +T  \|Vu\|_E\notag \\
         & \le  \sum_{ |\beta|\le \gamma} \|x^ \beta u_0\|_{ L^2( \R^d)} + \|u_0\|_{H^ \gamma(\R^d)}
       +CT \|u\|_E.
     \end{align}
By choosing $\M=2 \sum_{ |\beta|\le \gamma}\|x^ \beta u_0\|_{L^2(\mathbb{R}^d)} +2 \|u_0\|_{H^ \gamma(\R^d)}$ and sufficiently small $T$ such that $CT<1/2$, we have  
     \begin{align*}
          \| \Phi(u)\|_{E} 
            \le (1/2+ CT) \M \le \M.
     \end{align*}
This implies that $\Phi(u)\in E$. Next, we consider $u,v\in E$ and 
     \begin{align*}
     \Phi(u)-\Phi(v) =-i\int_0^t \tstd{i(t-s)} (V(x)(u(x,s)-v(x,s)))\,ds.
     \end{align*}       
By applying the same argument as in \eqref{est:non} to the identity above, we obtain
     \begin{align*}
      \rho(\Phi(u),\Phi(v)) \le 
      T \|V(u-v)\|_E
      \le CT \rho(u,v).
     \end{align*} 
Therefore, $\Phi:E\rightarrow E$ is a contraction if $T$ is sufficiently small. 

This proves the local well-posedness of equation \eqref{model-lsp} in $E$. Since $ T $ is independent of $u $, we can repeatedly apply this result on intervals $[T, 2T]$, $[2T, 3T]$, ..., to prove that $u\in C(\R; L^2(\R^d;|x|^{2 \gamma} d x))$. This completes the proof of Lemma \ref{lem:lwp}.
      
\hfill\endproof

\section*{Acknowledgments}
This work was supported in part by the Research Grants Council of the Hong Kong Special Administrative Region, China (Project No. PolyU/RFS2324-5S03, PolyU/GRF15306123) and an internal grant of Hong Kong Polytechnic University (Project ID: P0045404). 
\bigskip

\bibliographystyle{abbrv}
\bibliography{_LSUnboundedDomain}

\end{document}